%% file: Oscillation_v7.tex
\documentclass[11pt,reqno]{amsart}
\usepackage{amsmath,amssymb,amsthm,amscd,amsfonts}
\usepackage{graphicx}%
\usepackage{fancyhdr}
\usepackage{enumerate}
\usepackage{pdfsync}
\usepackage{subfigure}
\usepackage{enumerate}
\usepackage{tikz-cd}

\usepackage{color}

\theoremstyle{plain} \numberwithin{equation}{section}
\newtheorem{thm}{Theorem}[section]
\newtheorem{cor}[thm]{Corollary}

\newtheorem*{mal}{Mapping Analysis Lemma}
\newtheorem{lem}[thm]{Lemma}
\newtheorem{prop}[thm]{Proposition}
\newtheorem{clm}[thm]{Claim}

\theoremstyle{definition}
\newtheorem{defn}[thm]{Definition}

\newcommand{\diam}{\operatorname{diam}}
\newcommand{\proj}{\operatorname{proj}}
\newcommand{\cS}{\mathcal S^1}

\newcommand{\bbS}{\mathbb S^1}
\newcommand{\cl}{\operatorname{cl}}

\newcommand{\hull}{\operatorname{Hull}}
\newcommand{\Ball}{{\mathrm{B}}}

\renewcommand{\d}{\operatorname{dist}}
\newcommand{\im}{\operatorname{Im}}

\begin{document}

\title{Homotopy type of planar continua}
\author{Curtis Kent}

\begin{abstract}
We prove that the homotopy type of a map from a Peano continuum into a planar or one-dimensional space is determined by the induced homomorphism of fundamental groups. This provides a new proof that planar sets are aspherical and is used to prove that two planar or one-dimensional Peano continua are homotopy equivalent if and only if they have isomorphic fundamental groups.
\end{abstract}

\maketitle


\section{Introduction}

A central theme in topology is to find invariants and determine to what extent these invariants classify certain categories of topological spaces.  Poincar\'{e} defined the fundamental group and showed that it classified closed surfaces \cite{Poincare1895}.  He then asked to what extent this would hold for manifolds of higher dimension.  While the fundamental group is not sufficient to classify all higher dimensional manifolds, Whitehead proved  that  any two aspherical CW-complexes are homotopy equivalent exactly when they have isomorphic fundamental groups \cite{Whitehead_I49,Whitehead_II49}.



The standard techniques used to construct homotopy equivalences for CW-complexes are inadequate when considering spaces with local topological obstructions.  In \cite{eda}, Eda used the rigidity of homotopy classes of loops in one-dimensional spaces to prove that Whitehead's Theorem holds for one-dimensional Peano continua that are not locally simply connected at any point and later for all one-dimensional Peano continua in \cite{eda7}.

Here we will consider the category of planar and one-dimensional Peano continua and prove the following.

\begin{thm}\label{B}  Let $X$ and $ Y$ be planar or one-dimensional Peano continua.  The fundamental groups of $X$ and $Y$ are isomorphic if and only if $X$ and $Y$ are homotopy equivalent.\end{thm}

When $X$ and $Y$ are both one-dimensional, Theorem \ref{B} provides an alternative proof of \cite[Theorem 1.1]{eda7}.  Our proof is based on the following result that classifies the homotopy type of a map in terms of the induced homomorphism on fundamental groups.

\begin{thm}\label{A}
    The homotopy type of a continuous map from a Peano continuum into a planar or one-dimensional Peano continuum is determined by the induced homomorphism on fundamental groups.
\end{thm}

For a precise statement of Theorem \ref{A}, see Theorem \ref{homotopic maps}.  An immediate consequence of this is an alternate proof that planar and one-dimensional sets are aspherical.

\begin{cor} Planar and one-dimensional sets are aspherical.  \end{cor}

One-dimensional spaces were shown to be aspherical in \cite{CurtisFort57} and it was long believed that planar sets were aspherical but a proof proved elusive.  The first complete proof was given by Zastrow in \cite{Z}.  Cannon, Conner, and Zastrow later provided a more concise proof by considering a total oscillation function on the space of paths \cite{ccz}.

The total oscillation function can be defined on any subset of $n$-dimensional Euclidean space and measures how many times a path crosses a dense set of parallel hyperplanes.  A path with minimal total oscillation in a homotopy class will be called an \emph{oscillatory geodesic}.  The existence of oscillatory geodesics for one-dimensional spaces follows from the existence of reduced paths and they vary continuously with their endpoints.

Demonstrating the existence of oscillatory geodesics for planar continua and that they vary continuously with their endpoints, see Proposition \ref{continuously}, requires a careful study of the nature of planar homotopies, which will encompass the majority of Section \ref{pog}.  With these properties in hand, we show that any two functions, which induce conjugate homomorphisms of the fundamental group, are homotopic by homotoping along oscillatory geodesics, proving Theorem \ref{A}.  This together with the fact that homomorphisms of the fundamental group of planar or one-dimensional continua are induced by continuous maps, see Proposition \ref{inverse}, will complete the proof of Theorem \ref{B}.

\section{Definitions and the delineation map}

We give the following standard definitions to fix notation.

\begin{defn}
    We will denote the unit disc in the plane by $\mathbb D$ and its boundary by $\bbS$.  We will denote the open ball about $x$ of radius $r$ by $B_r^X(x) = \{ y\in X \ | \ d(x,y)<r\}$ and the sphere about $x$ of radius $r$ by $S_r^X(x)= \{ y\in X \ |\ d(x,y) = r\}$.  If $X$ is a planar set, then $B_r^X(x) = B_r^{\mathbb R^2}(x)\cap X$ and $S_r^X(x) = S_r^{\mathbb R^2}(x)\cap X$.  For $A\subset X$ and $\epsilon>0$, let $\mathcal N_\epsilon(A) = \bigl\{x\in A\mid d(x,A)<\epsilon\bigr\}$.  We will denote the topological closure in $X$ of a subset $U$  by $\cl_X(U)$ or simply $\cl(U)$, when $X$ is understood.  A subset of a metric space is \emph{thick} if it is not contained in a bounded neighborhood of its boundary.  A subset of a metric space that is not thick is \emph{thin}.

    We will use $f : (Y,y_1,y_2) \to (X,x_1,x_2)$ to denote a continuous map from $Y$ to $X$ such that $f(y_i) = x_i$ for $i= 1,2$.


    The property of being homotopic relative to endpoints defines an equivalence relation on the set of paths in $X$.  The equivalence class relative to endpoints of a path $\alpha$ will be denoted by $[\alpha]$ and will be called the \emph{homotopy class} of $\alpha$.  The equivalence class under free homotopies of a loop will be called the \emph{free homotopy class}.  To avoid confusion when considering paths, we will always use homotopic to mean homotopic relative to endpoints and never to mean freely homotopic.

\end{defn}

\begin{defn}
    Let $k$ be a line in the plane and $X$ a planar set. Let $\pi_k: X \to X_k$ be a decomposition map whose nontrivial decomposition elements are the maximal line segments contained in $X$ that are parallel to $k$.  We  refer to $\pi_k$ as a \emph{delineation map}.
\end{defn}

The following is Theorem 1.4 in \cite{cc4}.

\begin{thm}
    If $X$ is a planar continuum and $k$ is any line in the plane, then $\pi_k$ defines an upper semicontinuous decomposition of $X$, $X_k$ is a one-dimensional continuum, and the induced homomorphism of fundamental groups is injective.  If $X$ is a Peano continuum, then so is $X_k$.
\end{thm}

\begin{defn}
    A path $f:[0,1] \to X$ into a one-dimensional space is \emph{reducible} if there is an open interval $U\subset[0,1]$ such that $f|_{\cl{(U)}}$ is a nondegenerate nullhomotopic loop.  A path $f: [0,1] \to X$ is \emph{reduced} if it is not reducible.  A \emph{reduced representative} of a path $g$ is any reduced path homotopic relative to endpoints to $g$.  A constant path is, by definition, reduced and it is immediate that every reparametrization  of a reduced path is reduced.
 \end{defn}

 Curtis and Fort proved the following theorem in \cite[Lemma 3.1]{CurtisFort59} and another proof can be found in \cite[Theorem 3.9]{cc3}.

\begin{thm} \label{reducedloop-thm}
Let $f:[0,1]\to X$ be a path in a one-dimensional space $X$.  Then there exists an open set $\mathcal U\subset[0,1]$ such that $f|_{\cl(U)}$ is a nullhomotopic loop for each component $U$ of $\mathcal U$ and the path obtained by replacing $f|_{\cl(U)}$ with the constant path for each component $U$ of $\mathcal U$ is a reduced representative of $f$. In addition, any reduced representative of $f$ is a reparametrization of this path.
\end{thm}

\begin{defn}
    Let $\mathcal U$ be an open set satisfying the conclusion of Theorem \ref{reducedloop-thm} for $f$.  If  no two components of $\mathcal U$ share an endpoint, we will say that $\mathcal U$ is an \emph{open cancellation for $f$}.  The proof of Theorem \ref{reducedloop-thm} in \cite{cc3} shows that an open cancellation exists for every path in a one-dimensional space.
\end{defn}

The following lemma is used to construct homotopies by modifying a path on a sequence of subpaths.

\begin{lem}\label{aalem1}
Let $H:I\times I \to Z$ be a function into a metric space $Z$ such that $H|_{I\times\{0,1\}}$ is continuous.  Let $\bigl\{(a_i,b_i)\bigr\}$ be a set of disjoint open intervals in $I$ such  that $\diam\bigl(H([a_i,b_i]\times I)\bigr)$  converges to $0$.  Suppose that $H$ is constant on  $\{t\}\times I$  for $t\not\in\cup_i(a_i,b_i)$ and $H$ is continuous on each $[a_i,b_i]\times I$.  Then $H$ is continuous. \end{lem}

\begin{proof}
Consider a sequence $(x_n, y_n) \to (x_0,y_0)$ and let $\mathcal U = \bigcup\limits_i(a_i,b_i)$.  If $x_0\in \mathcal U $, then eventually $(x_n,y_n)\subset (a_i,b_i)\times I$ for some $i$.  Hence $H(x_n,y_n)$ converges to $H(x_0,y_0)$, since $H$ is continuous on each $[a_i,b_i]\times I$.  Thus we may assume that $x_0\not\in \mathcal U$.

Suppose that $x_{n_m}$ is a subsequence of $(x_n)$ such that $x_{n_m}\not\in \mathcal U$ for all $m$.  Then $H(x_{n_m}, y_{n_m})$ converges to $H(x_0,y_0)$ since $H$ is constant on $\{t\}\times I$ for $t\not\in \mathcal U$ and $H|_{I\times \{0,1\}}$ is continuous.

Suppose that $x_{n_m}$ is a subsequence of $(x_n)$ such that $x_{n_m}\in (a_{i_m}, b_{i_m})$ for all $m$.  Then there exists $c_m\in\{a_{i_m},b_{i_m}\}$ such that $\d(x_{n_m}, x_0) \geq \d(c_m, x_0)$. Hence $(c_m, y_{n_m})$ converges to $(x_0, y_0)$.  Applying the previous paragraph, we see that $H(c_m, y_{n_m})$ converges to $H(x_0, y_0)$.  Since $\diam\bigl(H([a_i,b_i]\times I)\bigr)$  converges to $0$, we see that $\d\bigl(H(c_m, y_{n_m}),H( x_{n_m}, y_{n_m}) \bigr)$ also converge to $0$.  Thus $H(x_{n_m}, y_{n_m})$ converges to $H(x_0,y_0)$.  Therefore $H$ is continuous. \end{proof}

The following is an consequence  of \cite[Lemma 13]{FischerZastrow2005}.

\begin{lem}\label{cutoff}
Suppose that $f:\mathbb S^1 \to X$ is a nullhomotopic loop in a planar set $X$.  Then $f$ is nullhomotopic in the $B^X_{r}\bigl(f(0)\bigr)$ for every $r> \diam\bigl(\im{(f)}\bigr)$.
\end{lem}

The following lifting lemma was originally proved in \cite{ConnerKentpreprint}.  We reproduce a proof here for completeness.

\begin{lem}[Reduced Path Lifting] \label{kreduced}
Let $X$ be a planar continuum and $k$ a line in the plane.  For every path $\alpha:I\to X$ there exists a path $\tilde \alpha:I\to X$ homotopic to $\alpha$ such  that $\pi_k\circ \tilde \alpha $  is  a reduced representative of $\pi_k\circ \alpha$.
\end{lem}

The induced map $\pi_{k*}$ on fundamental groups, in general, will not be surjective.  Thus, the lemma says that if the decomposition map hits a path, then one can lift a reduced representative of the path class.

\begin{proof}
    If $\pi_k\circ \alpha$ is reduced, we are done.  Otherwise, there exists an interval $[c,d]$ such that $\pi_k\circ \alpha\bigl|_{[c,d]}$ is a nonconstant nullhomotopic loop.  Then $\pi_k\circ \alpha(c) = \pi_k\circ \alpha(d)$ which implies that the line segment $\overline{\alpha(c)\alpha(d)}$ is contained in $X$.  Let $l$ be a parametrization of the line segment from $\alpha(c)$ to $\alpha(d)$.  The loop $\alpha\bigl|_{[c,d]} * \overline{l}$ maps to $\pi_k\circ \alpha\bigl|_{[c,d]}$ and hence must be nullhomotopic since $\pi_{k*}$ is injective. Therefore $\alpha$ is homotopic to $ \alpha'$ where the subpath $\alpha\bigl|_{[c,d]}$ is replaced by the path $l$.

    Let $\mathcal U $ be an open cancellation for $\pi_k\circ \alpha$ and $\bigl\{(c_i,d_i)\bigr\}_{i\in J}$ be the components of $\mathcal U$.  For each $i\in J$, let $l_i:[c_i,d_i] \to X$ be a parametrization of the line segment from $\alpha(c_i)$ to $\alpha(d_i)$.  Let  $\tilde \alpha$ be the path obtained by, for each $i\in J$, replacing the subpath $\alpha|_{[c_i,d_i]}$ by $l_i$. Then $\tilde \alpha$ defines a continuous path such that $\alpha(t) =\tilde \alpha(t)$ for $t\not\in \mathcal U$.

    We need to show that $\alpha$ is homotopic to $\tilde \alpha$.  Since $\alpha$ is uniformly continuous, $\diam{\bigr(\alpha \bigl|_{[c_i,d_i]}\bigl )}$ must converge to zero.

    As explained in the first paragraph of the proof, there exists continuous maps $H_i:[c_i,d_i]\times I \to X$ such that $H_i\bigr|_{[c_i,d_i]\times \{0\}} = \alpha\bigr|_{[c_i,d_i]}$,  $H_i\bigr|_{[c_i,d_i]\times\{1\}}=l_i$, and $H_i(c_i, t) = \alpha(c_i)$, $H_i(d_i,t) = \alpha(d_i)$.  By Lemma \ref{cutoff}, we may assume that $\diam{\bigl(H_i(I\times[c_i,d_i])\bigr)} \to 0$.  Define $H:I \times I \to X$ by $H(s,t) = H_i(s,t)$ if $s\in [c_i,d_i]$ and $H(s,t) = \alpha(s)$ otherwise.

    Lemma \ref{aalem1} implies that $H$ is continuous and thus $\alpha$ is homotopic to $\tilde \alpha$.  By our choice of $\mathcal U$, $\pi_k\circ\tilde \alpha$ is reduced which completes the proof.
\end{proof}

The following is a consequence of the previous proof, which we will formally state for later reference.

\begin{lem}\label{shortening}
    Let $\alpha: I\to X$ be a path into a planar continuum. Then $\alpha$ is homotopic to the line segment between its endpoints if and only if $\pi_k\circ\alpha$ is nullhomotopic where $k$ is a line passing through the endpoints of $\alpha$.
\end{lem}

\begin{defn}
    If $\alpha$ maps to a reduced path under $\pi_k$, we will say that $\alpha$ is \emph{reduced} with respect to $k$ or that $\alpha$ is \emph{$k$-reduced}.  For any path $\alpha$; if $\tilde \alpha$ is obtained from $\alpha$ as in Lemma \ref{kreduced}, i.e. $\tilde \alpha$ is obtained by making $\alpha$ linear on the components of an open cancelation of $\pi_k\circ\alpha$, then we will say $\tilde \alpha$ is obtained by \emph{reducing $\alpha$ with respect to $k$.}
\end{defn}

\section{Total oscillation function}

The following definition was originally given by Cannon, Conner, and  Zastrow in \cite{ccz}.  While we will follow the convention of \cite{ccz} and define the total oscillation function on any subset of $\mathbb R^n$, our primary interest will be the total oscillation function for planar sets and one-dimensional subsets of $\mathbb R^3$.

\begin{defn}[Oscillation]
Let $g: \mathbb R^n \to [-1,1]$ be given by


$$g(x_1, \cdots, x_n) = \left.\begin{cases}   ~-1 &\mbox{ if } ~~~~~~~~~x_1\leq -3 \\ \frac12(x_1+1) &\mbox{ if } -3\leq x_1\leq -1 \\ ~~~~0 &\mbox{ if } -1\leq x_1\leq 1\\  \frac12(x_1-1) &\mbox{ if } ~~~~1\leq x_1\leq 3  \\ ~~~~1 &\mbox{ if } ~~~~3\leq x_1 \end{cases}\right\}.$$

For any pair of parallel hyperplanes $(P,Q)$ in $\mathbb R^n$, there exists a homeomorphism $h_{(P,Q)}$ which is a composition of a dilation, a rotation, and a translation that sends $P$ to the hyperplane $\{x_1=-3\}$ and the $Q$ to the hyperplane  $\{x_1=3\} $.   Define $g_{(P,Q)}: \mathbb R^n \to [-1,1]$ by $g_{(P,Q)} =g\circ h_{(P,Q)}$.  While $h_{(P,Q)}$ is not uniquely defined, $g_{(P,Q)}$ is independent of the choice of the affine homeomorphism $h_{(P,Q)}$.   We will at times refer to $g_{(P,Q)}^{-1}(0)$ as the \emph{middle third of $(P,Q)$}.

The function $g_{(P,Q)}$ will be used to count \emph{continuously} the number of times a path oscillates between the hyperplanes $P$ and $Q$.  Given a continuous function $f:[a,b]\to \mathbb R^n$, define the \emph{oscillation} $o\bigl(f,(P,Q)\bigr)$ of $f$ with respect to the pair $(P,Q)$ to be supremum of the sum

$$-\sum\limits_{i=1}^k g_{(P,Q)} \bigl(f(a_{i-1})\bigr)\cdot g_{(P,Q)}\bigl(f(a_i)\bigr)$$

where the supremum is taken over all finite collections of points $a\leq a_0\leq a_1\leq\cdots\leq a_k\leq b$.  By way of convention, we will always take the sum to be 0 when considering collections containing 0 points or 1 point.   Thus $o\bigl(f, (P,Q)\bigr)$ is always non-negative.

\end{defn}

\begin{prop}\label{three parameters}[\cite{ccz}]
    The three parameters $f$, $P$, and $Q$ can all be parameterized by the compact open topology.  Then $o\bigl(f,(P,Q)\bigr)$ is a continuous function of all three parameters $f$, $P$, and $Q$. There is a finite collection of points $\{a_i\}$ for which the sum realizes the supremum.
\end{prop}

The finiteness of the collection follows from the fact that a continuous path can only go between two disjoint closed half-spaces finitely many times.    We state the following immediate corollary for compact subsets of $\mathbb R^n$.

\begin{cor}\label{small partial oscillation difference}
    Let $X$ be a compact subset of $\mathbb R^n$.  For every  path $\alpha:[0,1]\to X$ and $\epsilon>0$; there exists a $\delta>0$ such that if  $\beta: [0,a]\to X$ is a path such that $\d\bigl(\alpha(t),\beta(t)\bigr)< \delta$ for all $t$ then $\bigl|o\bigl(\alpha,(P,Q)\bigr)- o\bigl(\beta,(P,Q)\bigr)\bigr|<\epsilon$
\end{cor}

\begin{lem}\label{line segments}
  Let $X\subset \mathbb R^n$.  Suppose that $\alpha : [0,1]\to X$ is a parametrization of a line segment and $\beta:\bigl([0,1], 0,1\bigr)\to \bigl(X, \alpha(0), \alpha(1)\bigr)$ is any path. Then $o\bigl(\alpha, (P,Q)\bigr) \leq o\bigl(\beta, (P,Q)\bigr )$ for any $(P,Q)$.  If in addition, $\beta$ is not a reparametrization of $\alpha$ there exists $(P,Q)$ such that $o\bigl(\alpha, (P,Q)\bigr) < o\bigl(\beta, (P,Q)\bigr )$.
\end{lem}

\begin{proof}
    Let $\alpha : [0,1]\to X$ be a parametrization of a line segment and $\beta:\bigl([0,1],0,1\bigr)\to \bigl(X, \alpha(0),\alpha(1)\bigr)$ be any path.

    Fix disjoint parallel hyperplanes $P,Q$ in $\mathbb R^n$.  If $o\bigl(\alpha, (P,Q)\bigr) =0$, then the first conclusion is trivial since the oscillation function is non-negative.  Thus we may assume that $\alpha(0),\alpha(1)$ lie in distinct components of $\cl_X\bigl(X\backslash g^{-1}_{(P,Q)}(0)\bigr)$.  Then $o\bigl(\alpha, (P,Q)\bigr) =-g_{(P,Q)} \bigl(\alpha(0)\bigr)\cdot g_{(P,Q)}\bigl(\alpha(1)\bigr) = -g_{(P,Q)} \bigl(\beta(0)\bigr)\cdot g_{(P,Q)}\bigl(\beta(1)\bigr) \leq o\bigl(\beta, (P,Q)\bigr )$.

    Suppose that $\beta$ is not a reparametrization of $\alpha$.  Then we must find $(P,Q)$ such that $o\bigl(\alpha, (P,Q)\bigr) < o\bigl(\beta, (P,Q)\bigr )$.  Let $l$ be the line in $\mathbb R^n$ containing $\alpha(0)$ and $\alpha(1)$.

    \textbf{Case 1:} $\im(\beta)$ is not contained in the line $l$.  Then there exists $t\in(0,1)$ such that $\beta(t)\not\in l$. Let $P,Q$ be disjoint parallel  hyperplanes each of which separates $\beta(t) $ from  $l$ in $\mathbb R^n$.  Then $o\bigl(\alpha, (P,Q)\bigr) =0 <2= -g_{(P,Q)} \bigl(\beta(0)\bigr)\cdot g_{(P,Q)}\bigl(\beta(t)\bigr) -g_{(P,Q)} \bigl(\beta(t)\bigr)\cdot g_{(P,Q)}\bigl(\beta(1)\bigr)\leq  o\bigl(\beta, (P,Q)\bigr )$.

    \textbf{Case 2:} $\im(\beta)$ is  contained in the line $l$.  Since $\beta$ is not a reparametrization of $\alpha$, $\beta$ contains a nondegenerate nullhomotopic subloop $\beta|_{[s,t]}$.  Then for any disjoint parallel hyperplanes $P,Q$ such that $o\bigl(\beta|_{[s,t]}, (P,Q)\bigr)> 0$, we have that $o\bigl(\alpha, (P,Q)\bigr) < o\bigl(\beta, (P,Q)\bigr )$.
\end{proof}

Since the oscillation function counts the number of times a path crosses the middle third between to parallel hyperplanes, the following is an exercise.

\begin{lem}\label{not increase}Let $X\subset \mathbb R^n$ and $\alpha: [0,1] \to X$ a continuous map.  Then $o\bigl(\tilde\alpha, (P,Q)\bigr) \leq o\bigl( \alpha, (P,Q)\bigr )$ for every $(P,Q)$ where $\tilde \alpha$ is the path obtained by replacing a subpath of $\alpha$ by a parametrization of the line segment between its endpoints.
\end{lem}

The following lemma relates being $k$-reduced with having minimal oscillation for planar sets.

\begin{lem}\label{reduced oscillation}
Let $X$ be a planar continuum.  A path $\alpha: [0,1]\to X$ is $k$-reduced if and only if $o\bigl(\alpha, (P,Q)\bigr) \leq o\bigl(\beta, (P,Q)\bigr )$ for every path $\beta$ in the homotopy class of $\alpha$ and every pair of lines $(P,Q)$ parallel to $k$.
\end{lem}

\begin{proof}
    Suppose that $\alpha : [0,1]\to X$ is $k$-reduced and $P,Q$ is a pair  of lines parallel to $k$.  Fix $0\leq a_0\leq a_1\leq\cdots\leq a_k\leq 1$ such that $$o\bigl(\alpha, (P,Q)\bigr)= -\sum\limits_{i=1}^k g_{(P,Q)} \bigl(\alpha(a_{i-1})\bigr)\cdot g_{(P,Q)}\bigl(\alpha(a_i)\bigr).$$
    Let $\beta :[0,1]\to X$ be any path homotopic to $\alpha$ relative to endpoints.  Since $\pi_k\circ \alpha$ is a reduced path, there exists $0\leq b_0\leq b_1\leq\cdots\leq b_k\leq 1$ such that $\pi_k\circ \alpha(a_i) = \pi_k\circ \beta(b_i)$.  Thus $\alpha(a_i) $ and $\beta(b_i)$ lie on a line parallel to $k$.  Since $k$ is parallel to $P$ and $g_{(P,Q)}$ is constant on lines parallel to $P$, $g_{(P,Q)}\bigl(\alpha(a_i)\bigr) = g_{(P,Q)}\bigl(\beta(a_i)\bigr)$.  Thus $o\bigl(\alpha, (P,Q)\bigr) \leq o\bigl(\beta, (P,Q)\bigr )$.

    Suppose that $\alpha$ is not $k$-reduced.  Then there exists $c,d\in [0,1]$ such that  $\pi_k\circ \alpha\bigr|_{[c,d]}$ is a nondegenerate nullhomotopic loop.  Hence there exists $b\in(c,d)$ and lines $P,Q$ parallel to $k$ such that $\alpha\bigl(\{c,d\}\bigr)$ and $\alpha(b)$ are in distinct thick components of $\mathbb R^2\backslash \{P,Q\}$.  Let  $\alpha': [0,1] \to X$ be the path obtained by replacing $\alpha|_{[c,d]}$ with the line segment $\overline{g(c)g(d)}$.  Thus $\alpha'$ crosses the middle third between $(P,Q)$ strictly few times than $\alpha$.  Thus $\alpha'$ is homotopic to $\alpha$ and $o\bigl(\alpha', (P,Q)\bigr)< o\bigl(\alpha, (P,Q)\bigr)$.
 \end{proof}

\begin{defn}
All pairs of parallel hyperplanes in the $\mathbb R^n$ can obtained by a composition of the following functions: a dilation (to fix a distance between hyperplanes), a rotation (to fix an orientation of the hyperplanes), and a translation. Thus the set of parallel hyperplanes may be topologized by the compact open topology, and the resulting space is separable metric. Hence, we may fix a countable dense set $\bigl\{(P_i,Q_i)\bigr\}$ of parallel hyperplanes. We define the \emph{total oscillation} of a path $\alpha : [a,b]\to \mathbb R^n$ by the following formula  $$\mathcal T(\alpha) = \sum\limits_{i=1}^\infty \left(\frac{1}{2^i}\right)\left[\frac{o\bigl(\alpha,(P_i,Q_i)\bigr)}{1+ o\bigl(\alpha,(P_i,Q_i)\bigr)}\right].$$

We say that a path $f: [a,b]\to X$ is an \emph{oscillatory geodesic} if it has minimum total oscillation in its homotopy class.
\end{defn}

In general subsets of $\mathbb R^n$ will not always have oscillatory geodesics in every path class.  For example, there does not exists an oscillatory geodesic from $(-1,0)$ to $(1,0)$ in $\mathbb R^2\backslash\bigl\{(0,0)\bigr\}$.  Even Peano continua will not necessarily have oscillatory geodesics, since any space which is not homotopically Hausdorff (contains homotopically essential loops which can be homotoped arbitrarily small) will have path classes without oscillatory geodesics.

\begin{thm}[{\cite[Theorem 2.3]{ccz}}]

The total oscillation of a path is a continuous function on the space of paths. It vanishes on a path if and only if the path is constant. If a path $\alpha$ is the concatenation of paths $\beta$ and $\gamma$, with $\gamma$ nonconstant, then $\mathcal T(\alpha) >\mathcal T(\beta)$. Thus every path $\alpha$ may be parameterized by total oscillation. This parametrization collapses constant subpaths to degenerate subpaths.

\end{thm}

We will say that a path $\alpha: [0,1]\to X$ is \emph{parameterized proportional to total oscillation} if $\mathcal T\bigl( \alpha|_{[0,t]}\bigr) = t\cdot \mathcal T(\alpha)$.

For compact subsets, this gives us the following corollary.

\begin{cor}\label{small oscillation difference}Let $X$ be a compact subset of $\mathbb R^n$.  For every  $\epsilon>0$ there exists a $\delta>0$ such that if  $\alpha,\beta: [0,a]\to X$ are two continuous paths and $\d\bigl(\alpha(t),\beta(t)\bigr)< \delta$ for all $t$ then $|\mathcal T(\alpha)- \mathcal T(\beta)|<\epsilon$.
\end{cor}

\begin{lem}\label{oscillatory-lreduced}
A path in a planar continuum is an oscillatory geodesic if and only if it is $l$-reduced for any line $l$ in the plane.
\end{lem}

\begin{proof}
If a path is $l$-reduced for any line $l$ in the plane, then it is an oscillatory geodesic by Lemma \ref{reduced oscillation}.

Suppose that $\alpha$ is not $l$-reduced for some line $l$.  Lemma \ref{reduced oscillation} implies that  $o\bigl(\alpha,(P,Q)\bigr) >o\bigl(\tilde \alpha,(P,Q)\bigr)$ for some parallel line pair $(P,Q)$ where $\tilde \alpha$ is the path obtained by reducing $\alpha$ with respect to $l$.  Since $\bigl\{(P_i,Q_i)\bigr\}$ is a dense set of line pairs, the continuity of the oscillation function implies that there exists and $i$ such that $o\bigl(\alpha,(P_i,Q_i)\bigr) >o\bigl(\tilde \alpha,(P_i,Q_i)\bigr)$.  Lemma \ref{not increase} together with the continuity of the oscillation function imply that  $o\bigl(\tilde\alpha,(P,Q)\bigr) \leq o\bigl( \alpha,(P,Q)\bigr)$ for all $(P,Q)$.  Thus $\alpha $ is not an oscillatory geodesic.

\end{proof}

Since being $l$-reduced is preserved under taking subpaths, we have the following corollary.

\begin{cor}\label{subpaths}
If a path  in a planar continuum is an oscillatory geodesic, then each of its subpaths is also an oscillatory geodesic.
\end{cor}

\begin{cor}
The set of oscillatory geodesics in a planar continuum is independent of the dense set of parallel line pairs chosen to define the total oscillation function.
\end{cor}

Since the property of begin $l$-reduced for every line $l$ in the plane is an isometry invariant, we also have the following.

\begin{cor}
The set of oscillatory geodesics in a planar continuum is invariant under isometries of the plane.
\end{cor}

If $\phi$ is an isometry of a planar continuum and $\alpha$ is an oscillatory geodesic, as noted, $\phi\circ\alpha$ is still an oscillatory geodesic.  However, in general,  $\mathcal T(\alpha) $ will no be equal to $ \mathcal T(\alpha\circ\phi)$.

The following follows from Lemma \ref{line segments}.

\begin{lem}\label{lineunique}
A parametrization of a line segment in the plane is the unique oscillatory geodesic in its path class (up to  reparametrization).
\end{lem}

\begin{lem}\label{small diameter}
    Let $X$ be a bounded subset of $\mathbb R^n$.  For every $\epsilon>0$,  there exists a $\delta>0$ such that if $\alpha:[0,1]\to X$ is a continuous path and $\mathcal T(\alpha)< \delta$ then $\diam\bigl(\im(\alpha)\bigr)\leq \epsilon$.
\end{lem}

\begin{proof}
    Let $\bigl\{(P_i,Q_i)\bigr\}$ be the parallel hyperplanes used to define $\mathcal T$ on $X$.  Let $\{a_i\}$ be a finite maximal $\epsilon/8$-separated set in $X$, i.e. $d(a_i, a_j) \geq \epsilon/8$ for all $i,j$ and for every $x\in X$ there exists an $i$ such that $\d(x, a_i)\leq \epsilon/8$.  (It is an exercise to show that every metric space has a maximal $\eta$-separated subset and since $X$ is a bounded subset of $\mathbb R^n$ any $\eta$-separated subset must be finite.)

    For every pair $\{a_i,a_j\}$ with the property that $\d(a_i,a_j)>3\epsilon/4$, we may choose a parallel hyperplanes $(P_{k(i,j)}, Q_{k(i,j)})$ such that $\d\bigl( a_i, \{P_{k(i,j)},Q_{k(i,j)}\}\bigr),\d\bigl( a_j, \{P_{k(i,j)}, Q_{k(i,j)}\}\bigr)> \epsilon/4$ and $a_i,a_j$ are contained in distinct thick components of $\mathbb R^n \backslash \{P_{k(i,j)},Q_{k(i,j)}\}$.

    Thus any path which comes within $\epsilon/4$ of both $a_i$ and $a_j$ must have oscillation at least $1/2^{k(i,j)+1}$ if $\d(a_i,a_j)> 3\epsilon/4$.  If a path $\alpha$ has diameter greater than $\epsilon$, there exists $a_i, a_j$ such that $\d\bigl(a_i, \im(\alpha)\bigr), \d\bigl(a_j, \im(\alpha)\bigr)\leq \epsilon/8$ and $\d(a_i, a_j)>3\epsilon/4$.  Thus $\mathcal T(\alpha)\geq 1/2^{k(i,j)+1}$.

    Let $N = \max\{k(i,j)\}$; then any $\delta < 1/2^{N+1}$ gives the desired conclusion.
\end{proof}

\section{Planar oscillatory geodesics}\label{pog}

Our goal in this section is to show that in a planar Peano continuum homotopy classes of paths have unique oscillatory geodesics that vary continuously with their endpoints.  This was done in \cite{ccz} for one-dimensional spaces.

\begin{defn}
For $x,y\in \mathbb R^2$, we will use $[x,y]$ to denote the unoriented line segment between $x$ and $y$ in the plane.  Recall that $\bbS $ is the unit circle in the plane.  Let $\cS$ be the boundary of the unit square $[0,1]\times[0,1]$.  For convenience, we will refer to the four maximal line segments of $\cS$ as the \emph{sides} of $\cS$.  When more specificity is required, we will refer to $[0,1]\times \{1\}$,  $[0,1]\times \{0\}$, $\{0\}\times[0,1]$, and $\{1\}\times[0,1]$ as the \emph{top}, \emph{bottom}, \emph{left side}, and \emph{right side} of $\cS$ respectively.

\end{defn}

\begin{defn}
    We will say that paths $\alpha, \beta$ are \emph{$l$-straight-line homotopic}, if their exists a homotopy, $h: [0,1]\times [0,1] \to X$, and reparametrizations $\theta_\alpha,\theta_\beta :[0,1]\to [0,1]$ such that $h(t,0) = \alpha\circ\theta_\alpha(t)$, $h(t,1) = \beta\circ\theta_\beta(t)$, $h(0,t) = \alpha(0)$, $h(1,t) = \alpha(1)$, and $h|_{\bigl[\{t\}\times [0,1]\bigr]}$ is a parametrization of a line segment parallel to $l$ for every $t\in[0,1]$.  By way of convention, we will allow line segments to be degenerate (i.e. points) and  a degenerate line segment is parallel to every line segment.
\end{defn}

\begin{lem}\label{straight line}
    Let $\alpha,\beta: [0,1]\to X$ be $l$-reduced paths in a planar continuum $X$ that are homotopic relative to endpoints.  Then there exist reparametrizations $\theta_\alpha,\theta_\beta: [0,1]\to [0,1]$ such that $\pi_l\circ\alpha\circ\theta_\alpha(t) = \pi_l\circ\beta\circ\theta_\beta(t)$ for all $t\in[0,1]$.  Additionally, $h:[0,1]\times[0,1] \to X$ by $h(s,t) = (1-t)\alpha\circ\theta_\alpha(s) + t\beta\circ\theta_\beta(s)$ is an $l$-straight-line homotopy from $\alpha$ to $\beta$.
\end{lem}

\begin{proof}
    Since $\alpha,\beta$ are $l$-reduced paths they both map to reduced paths in $X_l$.  Since reduced paths are unique up to reparametrizations, there exist reparametrizations $\theta_\alpha,\theta_\beta:[0,1]\to [0,1]$ such that $\pi_l\circ\alpha\circ\theta_\alpha(t) = \pi_l\circ\beta\circ\theta_\beta(t)$ for all $t\in [0,1]$.

    Then $\bigr[\alpha\circ\theta_\alpha(t),\beta\circ\theta_\beta(t)\bigr]\subset X$ for all $t\in[0,1]$  and is parallel to $l$.  It is then immediate that the desire homotopy maps into $X$ and is continuous.
\end{proof}

\begin{thm}\label{existence}
    Every homotopy class of paths in a planar continuum has an oscillatory geodesic representative.
\end{thm}

\begin{proof}
    Let $X$ be a planar continuum and fix $\alpha:\bigl([0,1],0,1\bigr) \to \bigl(X,x,y\bigr)$.  Let $\bigl\{(P_i,Q_i)\bigr\}$ be the countable dense set of lines in the plane used to define the total oscillation function.

    Inductively define $\alpha_i :\bigl([0,1],0,1\bigr)\to\bigl(X,x,y\bigr)$ as the path obtained by reducing $\alpha_{i-1}$ with respect to $P_i$ where $\alpha_0 = \alpha$.  Therefore $\alpha_i$ is obtained by replacing a sequence of subpaths of $\alpha_{i-1}$ by the straight line joining the endpoints of the subpaths.  Thus we may assume that any modulus of continuity of $\alpha$ will also serve as a modulus of continuity of $\alpha_i$ for all $i$.

    There exists a subsequence $i_n$ such that $\alpha_{i_n}$ converges uniformly to a path $\beta:\bigl([0,1],0,1\bigr)\to\bigl(X,x,y\bigr)$.   (It will be a consequence of Proposition \ref{unique} that the sequence $\alpha_n$ converges without actually needing to pass to a subsequence.)

    We need to show that $\beta$ is homotopic to $\alpha$.  Recall that $X_l$ embeds into $\mathbb R^3$ and $\pi_1\bigl(X_l, \pi_l\circ\alpha(0)\bigr)$ embeds in $\lim\limits_\leftarrow \pi_1\bigl(\mathcal N_{1/m} (X_l), \pi_l\circ\alpha(0)\bigr)$  where we consider the neighborhoods taken in $\mathbb R^3$.  Therefore a path is homotopic to $\alpha$ in $X_l$ if and only if it is homotopic to $\alpha$ in $\mathcal N_{1/m} (X_l)$ for all $m$.  Fixing $m$, there exists an $n(m)$ such that $\pi_l\circ \alpha_{i_{n(m)}}$ is homotopic to $\pi_l\circ\beta$ in $\mathcal N_{1/m} (X_l)$ by a straight-line homotopy.  Since this holds true for all $m$ and $\pi_l\circ \alpha_{i_{n(m)}}$ is homotopic to $\pi_l\circ\alpha$ in $X_l$; we have that $\pi_l\circ \alpha$ is homotopic to $\pi_l\circ\beta$.  Thus $\alpha$ is homotopic to $\beta$.

    We must now show that $\beta$ has minimal total oscillation in its homotopy class. Suppose that for some $\gamma$ homotopic to $\beta$ there exists an $i$ and $\epsilon>0$ such that $o\bigl(\beta, (P_i,Q_i)\bigr) - o\bigl(\gamma, (P_i,Q_i)\bigr)\geq \epsilon$.

    By Corollary \ref{small partial oscillation difference}, there exists a $\delta >0$ such that for all paths $\eta:[0,1]\to X$ with $d\bigl(\eta(t), \beta(t)\bigr)<\delta$ for all $t$, we have that $\bigl|o\bigl(\beta, (P_i,Q_i)\bigr) - o\bigl(\eta, (P_i,Q_i)\bigr)\bigr|< \frac{\epsilon}{2}$.  Thus for all $n$ sufficiently large, $\bigl|o\bigl(\beta, (P_i,Q_i)\bigr) - o\bigl(\alpha_{i_n}, (P_i,Q_i)\bigr)\bigr|< \frac{\epsilon}{2}$.

    By construction, $\alpha_{i_n}$ is $P_i$-reduced for $i\leq i_n$. Hence for all sufficiently large $n$, we have  $o\bigl(\alpha_{i_n}, (P_i,Q_i)\bigr) \leq  o\bigl(\gamma, (P_i,Q_i)\bigr)< o\bigl(\beta, (P_i,Q_i)\bigr)$.  Thus $\epsilon\leq o\bigl(\beta, (P_i,Q_i)\bigr) - o\bigl(\gamma, (P_i,Q_i)\bigr) \leq o\bigl(\beta, (P_i,Q_i)\bigr) - o\bigl(\alpha_{n_i}, (P_i,Q_i)\bigr)\leq \frac{\epsilon}{2}$ which is a contradiction.  Therefore $\beta$ has minimal total oscillation in its homotopy class.
\end{proof}

\begin{defn}
A line $l$ is \emph{transverse} to $f:[a,b]\to X$ if no nonempty open interval $(c,d)\subset[a,b]$ maps into a line segment parallel to $l$.  A line $l$ is \emph{transverse} to $f:\cS\to X$ if it is transverse to $f$ restricted to each side of $\cS$.
\end{defn}

The property that $l$ is transverse to $f: [a,b]\to X$ depends on more than just the path but also upon the parametrization of $f$. No line is transverse to any path which maps a nonempty open interval to a point.

\begin{lem}
If $\{\alpha_t: [0,1]\to X\}$ is a countable set of nondegenerate oscillatory geodesics in the plane which are parameterized proportional to total oscillation, then there exists a line $l$ which is transverse to $\alpha_t$ for all $t$.
\end{lem}

\begin{proof}  Let  $\bigl\{(a_i,b_i)\mid i\in J_t\bigr\}$ be a set of maximal nonempty open intervals of $[0,1]$ such that $\alpha_t$ maps each interval into a line segment in the plane.   Notice that $(a_i,b_i)\cap (a_j,b_j)= \emptyset$ for $i\neq j$, since the intervals were maximal.  Thus $J_t$ must be countable (a compact interval can only contain countably many disjoint open subintervals).  Since $\alpha_t$ is parameterized proportional to total oscillation, $\alpha_t\bigl((a_i,b_i)\bigr)$ maps into a unique nondegenerate line segment. Thus there are only countably many directions which are not transverse to $\alpha_t$ for a fixed $t$.

Since the set of paths is countable and the set of directions in the plane is uncountable the lemma follows.

\end{proof}

\begin{defn}
 Let $\proj_x, \proj_y: \mathbb R^2 \to \mathbb R$ be the projections of the Euclidean plane onto the $x$-axis and $y$-axis respectively.

\end{defn}

\begin{thm}\label{unique}
There is a unique, up to reparametrization, oscillatory geodesic representative for every homotopy class of paths in a planar continuum.
\end{thm}

\begin{proof}
    Let $X$ be a planar continuum.  Theorem \ref{existence} proved the existence of oscillatory geodesics; thus, we need only show uniqueness.  Suppose that $\alpha, \beta: \bigl([0,1], 0,1\bigr) \to \bigl(X, x_0,x_1\bigr)$ are oscillatory geodesics, which are homotopic relative to endpoints and parameterized proportional to total oscillation.

    Let $S$ be the set of points $s\in[0,1]$ with the property that there exists a $t\in[0,1]$ such that $\alpha(s) = \beta(t)$ and the paths $\alpha|_{[0,s]}$ and $\beta|_{[0,t]}$ are homotopic.  If the subpaths $\alpha|_{[0,s]}$ and $\beta|_{[0,t]}$ are homotopic, then $\mathcal T\bigl(\alpha|_{[0,s]}\bigr) =  T\bigl(\beta|_{[0,t]}\bigr)$ since both are oscillatory geodesics by Corollary \ref{subpaths}.  Thus $s\mathcal T(\alpha) = \mathcal T\bigl(\alpha|_{[0,s]}\bigr) =  T\bigl(\beta|_{[0,t]}\bigr)= t\mathcal T(\beta)$ which implies that $s= t$.

    Suppose that $s_n$ is a sequence of points in $S$ converging to $s$.  Then $\alpha(s) = \beta(s)$.  Since the paths $\alpha|_{[0,s_n]}$ and $\beta|_{[0,s_n]}$ are homotopic and planar sets are shape injective (see \cite{FischerZastrow2005}), $\alpha|_{[0,s]}$ and $\beta|_{[0,s]}$ are homotopic and $s\in S$.  Thus $S$ is a closed subset of $[0,1]$.  If $S=[0,1]$, then $\alpha=\beta$ and we are done.  Otherwise by passing to a complimentary component of $S$ in $[0,1]$, we may assume that for every $s,t\in [0,1]$ if $\alpha(s) = \beta(t) $ then $\alpha|_{[0,s]}$ and $\beta|_{[0,t]}$ are not homotopic.

     Lemma \ref{lineunique} implies that neither $\alpha$ nor $\beta$ parameterizes a line segment.  After composition with a rotation and reflection of the plane, we may assume that $\bigl[\alpha(0),\alpha(1)\bigr]$ is contained in the line $\{y=0\}$ and $\im(\alpha)\cup\im(\beta)$ intersects the open half space $\bigl\{(x,y)\in \mathbb R^2 \mid y>0\bigr\}$.

    Let $l$ be a non-horizontal line in the plane that is transverse to both $\alpha$ and $\beta$.  Lemma \ref{straight line} implies that by passing to reparametrizations of $\alpha$ and $\beta$, we may assume that  $\pi_l\circ\alpha(t) = \pi_l\circ\beta(t)$ for all $t\in[0,1]$.  By our choice of $\alpha$ and $\beta$, the new parameterizations of  $\alpha$ and $\beta$ have the property that $\alpha(t) \neq \beta(t)$ for all $t\in(0,1)$.  Thus we can assume, without loss of generality,  that $\proj_y\bigl(\alpha(s)\bigr)>\proj_y\bigl(\beta(s)\bigr)$ for all $s\in(0,1)$.

    Fix $t_0\in(0,1)$ such that $\proj_y\bigl(\alpha(s)\bigr)\leq \proj_y\bigl(\alpha(t_0)\bigr)$ for all $s\in [0,1]$. Since $\proj_y\bigl(\beta(s)\bigr) < \proj_y\bigl(\alpha(s)\bigr)$ for all $s\in[0,1]$, there exists $\epsilon >0$ such that $\proj_y\bigl(\beta(s)\bigr) \leq \proj_y\bigl(\alpha(t_0)\bigr) -\epsilon$ for all $s\in[0,1]$.  Let $k$ be the horizontal line $\{y = c\}$ where $c= \max\bigl\{\proj_y\bigl(\beta(s)\bigr) \mid s\in[0,1]\bigr\} +\epsilon/2$.  Then $\beta$ does not intersect $k$ but $\alpha(t_0)$ lies strictly above the line $k$ in the plane.  Let $(a,b)$ be the component of $\alpha^{-1}\bigl(\{(x,y)\mid y>c\}\bigr)$ which contains $t_0$.  Thus for all $s\in [a,b]$ the parallel line segments $[\alpha(s),\beta(s)\bigr]$ must intersect $k$  which implies that $\alpha|_{[a,b]}$ is homotopic to $\bigl[\alpha(a),\alpha(b)\bigr]$.  Since $\alpha|_{[a,b]}$ is not contained in the line segment $\bigl[\alpha(a),\alpha(b)\bigr]$, this contradicts the fact that $\alpha$ has minimal total oscillation.
\end{proof}

\begin{defn}\label{small endpoint homotopic}
    We will say that two paths $\alpha,\beta: [0,1]\to X$ are \emph{$\delta$-endpoint homotopic} in $X$, if there exists $h: [0,1]\times[0,1] \to X$ such that $h(t,0) = \alpha(t)$, $h(t,1) = \beta(t)$, and $\diam\bigl(\im(h|_{\{0\}\times [0,1]})\bigr), \diam\bigl(\im(h|_{\{1\}\times [0,1]})\bigr)<\delta$.

    We will say that paths $\alpha:[0,a]\to X$ and $\beta: [0,b]\to X$ are \emph{$\epsilon$-close} if, for all $t\geq 0$, $$\d\Bigl(\alpha\bigl(\min\{t,a\}\bigr),\beta\bigl(\min\{t,b\}\bigr)\Bigr)< \epsilon.$$
\end{defn}

\begin{defn}
    A nullhomotopic map, $f:\cS\to X$, is an \emph{oscillatory square} if $f$ restricted to each side of $\cS$ is a nondegenerate oscillatory geodesic parameterized proportional to total oscillation.  A \emph{pseudo-oscillatory square} is an oscillatory square where the left and right sides of $\cS$ are allowed to map degenerately.  The \emph{top, bottom}, and \emph{sides} of a (pseudo-)oscillatory square are the images of the top, bottom, sides of $\cS$.

    Suppose that $f: \cS \to X$ is a nullhomotopic loop and $s,t\in[0,1]$. We will say that an ordered pair $(s,t)$ \emph{subdivides} $f$, if $f$ extends to a map of $I\times I$ such that the line segment from $(s,1)$ to $(t,0)$ maps to $\bigl\{f\bigl((s,1)\bigr)=f\bigl((t,0)\bigr)\bigr\}$, i.e. maps degenerately.  In which case, we will call the ordered pair $(s,t)$ a \emph{subdivision} of $f$ and we will say that $f$ is \emph{thick} if it has no subdivisions.
\end{defn}

\begin{lem}\label{subdivisions}
    Let $f:\cS \to X$ be a pseudo-oscillatory square. If $f$ has a subdivision, then there exists subdivisions $(s_0,t_0),(s_1,t_1)$ such that $f$ has no subdivision $(s,t)$ with

    $$(s,t)\in \Bigl([0,s_0)\times [0,t_0)\Bigr) \cup \Bigl( (s_1,1] \times (t_1,1]\Bigr).$$
  \end{lem}

Note that if either $s_0$ or $ t_0$ is $0$, we are using the convention $[0,0)$ is empty.  Similarly we will assume that $(1,1]$ is also empty.

\begin{proof}
Let $(s_i, t_i)$ be a sequence of subdivisions of $f$ a pseudo-oscillatory square.  Suppose that $s_i$ converges to $s$ and $t_i$ converges to $t$.  Thus $f(s) = f(t)$.  As well, $f$ restricted to the subpath of $\cS$ between $(s_i,1)$ and $(t_i, 0)$ containing the right side of $\cS$ is a sequence of nullhomotopic loops converging to $f$ restricted to the corresponding subpath of $\cS$ between $(s,1)$ and $(t,0)$.  Thus $(s,t)$ subdivides $f$ since planar sets are shape injective.  Hence $\bigl\{ (s,t)\mid (s,t) \text{ subdivides } f\bigr\}$ is a closed subset of the plane.
Let
$$s_0=\inf\bigl\{ s\in[0,1]\mid (s,t) \text{ subdivides } f \text{ for some } t\in[0,1]\bigr\}.$$  Since the top and bottom of $f$ map to oscillatory geodesics, there exists a unique $t_0$ such that $(s_0,t_0)$ is a subdivision of $f$.  Let
$$s_1=\sup\bigl\{ s\in[s_0,1]\mid (s,t) \text{ subdivides } f \text{ for some } t\in[t_0,1]\bigr\}.$$  Let $t_1$ be the unique point such that $(s_1,t_1)$ subdivides $f$.


By the previous paragraph, $(s_0,t_0)$, $(s_1,t_1)$ subdivide $f$ and, by construction, satisfy the desired maximality condition.
\end{proof}

\begin{mal}\label{mal}[Cannon and Conner] Suppose that $f:\bbS\to M$ is a nullhomotopic mapping from the circle $\bbS$ into a one-dimensional continuum $M$. Then there is an upper
semicontinuous decomposition $H$ of $\bbS$ into compacta that has the following three properties:

\begin{enumerate}[(i)]
	
	\item The mapping $f$ is constant on each element of $H$.
	\item\label{malii} The decomposition $H$ is \emph{noncrossing}. That is,
if $h_1$ and $h_2$ are distinct elements of $H$, then
the convex hulls $\hull(h_1)$ and $\hull(h_2)$ of $h_1$ and
$h_2$ in the unit disc are disjoint. [Equivalently, $h_1$
does not separate $h_2$ on $\bbS$.]

	\item\label{maliii} The decomposition $H$ is \emph{filling}. That is,
the unit disc is the union of the convex hulls $\hull(h)$
of the elements $h\in H$.

\end{enumerate}
\end{mal}

We will refer the interested reader to \cite{cc4} for a complete proof.  However, we will give a sketch of how $H$ is constructed.

\begin{proof}[Sketch of proof]
Since $M$ is one-dimensional, the map $f$ factors through a map $f':\mathbb S^1 \to D$ where $D$ is a planar dendrite.  (A dendrite is a connected, locally connected, compact metric space that contains no simple closed curve.)  Then $H=\{f^{'-1}(x)\mid x\in D\}$.
\end{proof}

We require a version of the Mapping Analysis Lemma for planar sets to study how oscillatory geodesics in the plane vary with their endpoints. We will present the planar version in terms of oscillatory squares instead of the unit circle.  Requiring that sides of $\cS$ map nicely, i.e. are oscillatory geodesics, allows us to understand how elements of the upper semicontinuous decomposition intersect the sides of $\cS$.

\begin{lem}[Mapping Analysis Lemma for planar maps] \label{malPlanar} Suppose that $f:\cS\to X$ is a oscillatory square in a planar continuum and $l$ is a line transverse to $f$. Then there exists an upper semicontinuous decomposition $H$ of $\cS$ into compacta that has the following properties:

\begin{enumerate}[(i)]

	\item\label{malpi} The mapping $f$ takes each element of $H$ into a line segment in $X$ parallel to $l$.
	\item\label{malpii} The decomposition $H$ is \emph{noncrossing}. That is, if $h_1$ and $h_2$ are distinct elements of $H$, then the convex hulls $\hull(h_1)$ and $\hull(h_2)$ of $h_1$ and $h_2$ in the plane are disjoint.
	\item\label{malpiii}  The decomposition $H$ is \emph{filling}. That is,
the unit square is the union of the convex hulls $\hull(h)$
of the elements $h\in H$.
	\item\label{malpiv} Every element of $H$ intersects each side of $\cS$ in at most one point and no element of $H$ is contained in the interior of a side of $\cS$.

\end{enumerate}
\end{lem}

\begin{proof}
Fix $l$ a line transverse to $f$.  Let $g: \bbS \to \cS$ be an orientation preserving homeomorphism.  Let $H'$  be the upper semicontinuous decomposition obtained by applying the Mapping Analysis Lemma to $\pi_l\circ f\circ g: \bbS \to X_l$ where $X_l$ is the one-dimensional continuum obtained by the collapsing the maximal line segments in $X$ parallel to $l$.  Let $H = \{ g(h')\mid h'\in H'\}$.

Since $\pi_l\circ f\circ g$ is constant on each element of $H'$, $\pi_l\circ f$ is constant on each element of $H$.  This implies that $f$ maps each element of $H$ into a line segment in $X$ which is (\ref{malpi}).

\begin{clm}\label{clm1}
Suppose that $x,y$ are distinct points of an element $h'\in H'$.  Then $g(x), g(y)$ are not contained in the same side of $\cS$.
\end{clm}

\begin{proof}[Proof of claim]  If $x,y$ are distinct points in the same element of $H'$ such that $ g(x)$ and $g(y)$ are contained in the same side of $\cS$, then $f|_{\bigl[ g(x), g(y)\bigr]}$ is an oscillatory geodesic that, by Lemma \ref{shortening}, is homotopic to the interval $\bigl[ f\circ g(x), f\circ g(y)\bigr]$.  Thus $f|_{\bigl[ g(x), g(y)\bigr]}$ is a parametrization of the interval $\bigl[ f\circ g(x), f\circ g(y)\bigr]$ which contradicts the hypothesis that $l$ is transverse to $f$.
\end{proof}

By Claim \ref{clm1}, $H$ satisfies the first part of (\ref{malpiv}).  This implies that the only way an element of $H$ could be contained in the interior of a side of $\cS$ is if it is a single point of $\cS$.

From the proof of the Mapping Analysis Lemma, $\pi_l\circ f \circ g$ factors through a surjective map $\bar f : \mathbb S^1 \to D$ to a planar dendrite, $D$.  For any point $x$ such that $D\backslash \{x\}$ has two components, $|\bar f^{-1}(x)|>1$ since no point separates $\mathbb S^1$.  Recall $f$ restricted to each side of $\cS$ is an oscillatory geodesic parameterized proportional to total oscillation and $l$ is transverse to $f$ restricted to that side; thus $\bar f\circ g^{-1}$ must map each side of $\cS$ injectively into $D$.   Hence every point on the interior of a side of $\cS$ must have image which separates $D$ into at least two components.  Thus $H$ satisfies (\ref{malpiv}).

Property (\ref{malii}) of the Mapping Analysis Lemma implies that $h_1$
does not separate $h_2$ on $\cS$ for any $h_1,h_2\in H$ (not separating is preserve under homeomorphisms). Since each element of $H$ intersects each side of $\cS$ in at most one point, property (\ref{malpii}) holds for $H$.

The proof of property (\ref{maliii}) is exactly the same as the proof of the corresponding property for the Mapping Analysis Lemma for one-dimensional spaces.  We will sketch the proof here and refer the reader to \cite[p. 61-62]{cc4} for details.  The upper semicontinuous decomposition $H$ extends to a cellular decomposition $\tilde H$ of $\mathbb R^2$ whose non-degenerate elements are $\{\hull(h) \mid h\in H\}$.  Let $\pi: \mathbb R^2\to\mathbb R^2$ be the quotient map corresponding to $\tilde H$.  Notice that $\pi(\cS)$ is contractible, since it factors through a dendrite.  If $H$ wasn't filling then, then $\pi(\cS)$ would separate $\pi(\mathbb R^2\backslash [0,1]\times[0,1])$ and $\pi\bigl([0,1]\times[0,1]\bigr)\backslash \pi(\cS)$ which would be a contradiction.

\end{proof}

\begin{lem}\label{once}
    Let $f:\cS \to X$ be a thick oscillatory square in a planar continuum $X$.  Suppose that for some line $k$ in the plane $f|_{\{0,1\}\times [0,1]}$ maps into a single component of $\mathbb R^2\backslash k$.  Then $\im(f)$ is contained in the closure of a single component of $\mathbb R^2\backslash k$.
\end{lem}

\begin{proof}
Suppose that $k$ is a line in the plane such that $f|_{\{0,1\}\times [0,1]}$ maps into a single component of $\mathbb R^2\backslash k$.  After rotating the plane if necessary, we may assume that $k$ is a horizontal line.
Fix $l$, a line transverse to $f$ and not parallel to $k$. Let  $H$ be an upper semicontinuous decomposition of $\cS$ satisfying the conclusions of Lemma \ref{malPlanar} for $f$ and $l$.

By Condition (\ref{malpi}), each element of $H$ maps into an interval contained in $X$ which is parallel to $l$.  By Condition (\ref{malpiv}), every element of $H$ which intersects the interior of a side of $\cS$ must also intersect one of the other three sides of $\cS$.

Let $T$ be the component of $\mathbb R^2\backslash k$ which is disjoint from $\im\bigl(f|_{\{0,1\}\times [0,1]}\bigr)$.  Since the claim remains unchanged by reflecting the plane about $k$, we may assume that $\proj_y\bigl(f(i,a)\bigr)< \proj_y (k)$ for every $(i,a) \in \{0,1\}\times [0,1]$.

\begin{clm}
If every element of $H$ which intersects the top of $\cS$ also intersects $\{0,1\}\times[0,1]$, then $\im(f)$ is contained in the closure of a single component of $\mathbb R^2\backslash k$.
\end{clm}

\begin{proof}[Proof of claim.]

Suppose there exists an open interval $(a,b)$ contained in $[0,1]$ with the property that $f|_{(a,b)\times \{1\}}$ maps into $T$.  We may assume that $(a,b)$ is a maximal such interval.  For every $x\in (a,b)\times \{1\}$, the  line segment in $X$ parallel to $l$ containing $f(x)$ must intersect $k$, since $k$ separates $f\big((a,b)\times \{1\}\bigr)$ and $f\bigl(\{0,1\}\times[0,1]\bigr)$.  Notice that $f(a), f(b)\in k$.  Thus $f|_{[a,b]\times \{1\}}$ is $l$-straight-line homotopic to but not contained in the line segment $[f(a),f(b)]$, which contradicts our hypothesis that $f_{[0,1]\times\{1\}}$ was an oscillatory geodesic.

The hypothesis that every element of $H$ which intersects the top of $\cS$ also intersects $\{0,1\}\times[0,1]$ implies that   every element of $H$ which intersects the bottom of $\cS$ also intersects $\{0,1\}\times[0,1]$. So after applying the same argument to the bottom of $\cS$ the subclaim is proved.

\end{proof}

By way of contradiction, suppose that $\im(f)$ is not contained in the closure of a single component of $\mathbb R^2\backslash k$. Then there exists an element of $H$ which intersects only the top and bottom of $\cS$.

Since $H$ is noncrossing, there are a maximal intervals $(a_0,b_0), (a_1,b_1)\subset [0,1]$ and an orientation preserving homeomorphism $\sigma:  (a_1,b_1)\to (a_0,b_0)$ such that $\bigl\{(x,1),\bigl(\sigma(x),0\bigr)\bigr\}\in H$ for every $x\in(a_1,b_1)$.

\begin{figure}[h]\label{minimizing}
\centering
\def\svgwidth{\columnwidth}
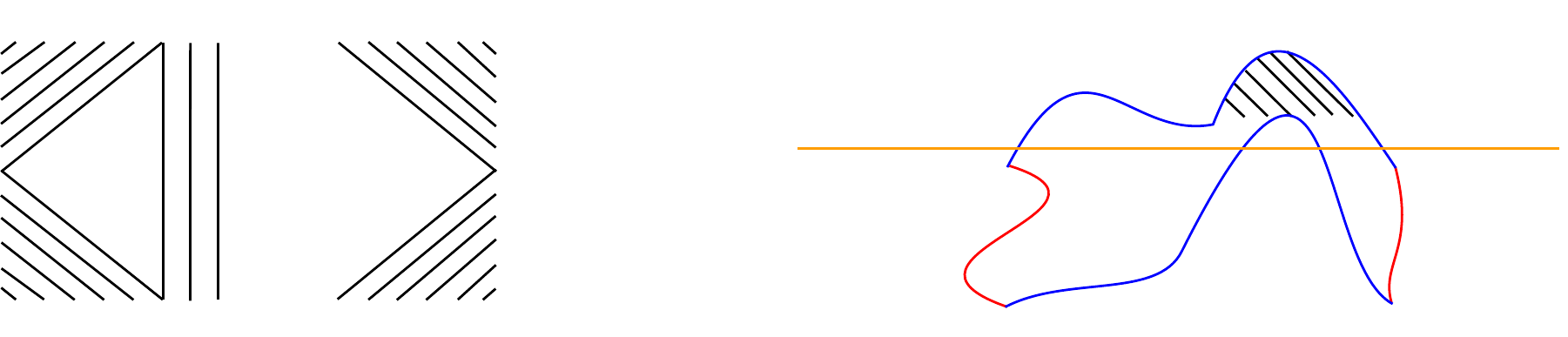\caption{}
\end{figure}

If $\proj_y\bigl( f( x,1) \bigr) =\proj_y\bigl( f (\sigma(x),0)\bigr)$ for some $x\in (a_1,b_1)$, then $f(x,1) = f(\sigma(x),0)$ (they both lie on non-horizontal line) and $\bigl(x, \sigma(x)\bigr)$  would subdivide $f$. Thus, since $f$ is thick, $\proj_y\circ f$ is non-constant on $\bigl\{(x,1),\bigl(\sigma(x),0\bigr)\bigr\}$ for all $x\in(a_1,b_1)$ which implies that one of the following two inequality holds.

\begin{equation}\proj_y\bigl( f (x,1) \bigr) < \proj_y\bigl( f (\sigma(x),0)\bigr)\text{ for all } x\in (a_1,b_1)\end{equation} or

\begin{equation}\label{aa}\proj_y\bigl( f (x,1) \bigr) > \proj_y\bigl( f (\sigma(x),0)\bigr)\text{ for all } x\in (a_1,b_1)\end{equation}

We will assume (\ref{aa}) holds.  The other case is similar.

Let $\delta = \max \Bigl\{\proj_y\bigl(f(x,0)\bigr) - \proj_y\bigl(k\bigr)\mid x\in[a_0,b_0]\Bigr\}$, i.e. the maximal amount that $f$ restricted to the bottom of $\cS$ goes above the line $k$. Let $\bar k$ be the horizontal line $k$ if $\delta<0$ or the positive vertical translation of $k$ by $\delta$ if $\delta\geq 0$.

Notice that by our choose of $\bar k$, $\im(f)$ is not contained in the closure of a single component of $\mathbb R^2 \backslash\bar k$.  Let $\overline T$ be the component of $\mathbb R^2\backslash \bar k$ which is disjoint from $\im\bigl(f|_{\{0,1\}\times [0,1]}\bigr)$.  Let $(a,b)$ be the maximal subinterval of $[a_1,b_1]$ such that $f|_{(a,b)\times \{1\}}$ maps into $\overline T$.

By our choice of $\bar k$, the  line segment in $X$ parallel to $l$ starting at $f(x)$ must intersect $\bar k$ for every $x\in (a,b)\times \{1\}$.  Notice that $f(a), f(b)\in \bar k$.  Thus $f|_{[a,b]\times \{1\}}$ is $l$-straight-line homotopic to but not contained in the line segment $[f(a),f(b)]$ which contradicts our hypothesis that $f_{[0,1]\times\{1\}}$ was an oscillatory geodesic, which completes the proof of the lemma.

\end{proof}

\begin{lem}\label{thin}
Let $f:\cS \to X$ be a thick oscillatory square in a planar continuum $X$. Suppose that $k_N, k_S$ are parallel lines in the plane such that $\im(f|_{\{0,1\}\times[0,1]})$ is contained in the thin component of $\mathbb R\backslash \{k_N\cup k_S\}$.  Then $\im(f)$ is contained in the thin component of $\mathbb R\backslash \{k_N\cup k_S\}$.

\end{lem}

\begin{proof}

Choose two lines $\bar k_N$ and $\bar k_S$ such that $f\bigl(\{0,1\}\times[0,1]\bigr)$ is still contained in the thin component of $\mathbb R^2 \backslash \bigl\{\bar k_N\cup \bar k_S\bigr\}$ and $\bar k_N\cup \bar k_S$ is contained in the thin component of $\mathbb R^2 \backslash \bigl\{ k_N\cup k_S\bigr\}$.  Applying Lemma \ref{once} to each of $\mathbb R^2 \backslash \bar k_N$ and  $\mathbb R^2 \backslash \bar k_S$ completes the proof.

\end{proof}

We now want to prove a version of Lemma \ref{thin} which holds for pseudo-oscillatory geodesic squares.


\begin{lem}\label{reparameterization}
Let $f:\cS \to X$ be a pseudo-oscillatory square in a planar continuum $X$ and $\delta>0$.  There exists an oscillatory square $\bar f: \cS \to X$  and a monotone quotient map $\psi: \cS \to \cS$ such that $\bar f\circ\psi = f$ and $\d\bigl(f(x),\bar f(x)\bigr)\leq \delta$ for all $x\in\cS$.  If in addition $f$ is thick, then $\bar f$ is also thick.

\end{lem}

\begin{proof}

Fix $\delta >0$ and choose $\delta_1\in(0,1/3)$ such that  $d\bigl(f(x), f(y)\bigr)< \delta$ for all $x,y \in \cS$ with $d(x,y)< 2\delta_1$.  Then we can define $\psi$ as follows.

The map $\psi$ is the identity on $\bigl([0,1]\times \{0\} \bigr)\cup\bigl( [\delta_1, 1-\delta_1]\times \{1\}\bigr)$.

If the right side of $f$ is degenerate, then $\psi$ collapses the right side of $\cS$ to $(1,0)$ and sends $[1-\delta_1, 1]\times\{1\}$ to $\bigl([1-\delta_1, 1]\times\{1\}\bigr) \cup\bigl( \{1\}\times [0,1]\bigr)$ by the an appropriate piecewise linear homeomorphism fixing $(1-\delta_1, 1)$. Otherwise $\psi$ is the identity on $\bigl([1-\delta_1, 1]\times\{1\}\bigr) \cup\bigl( \{1\}\times [0,1]\bigr)$.

If the left side of $f$ is degenerate, then $\psi$ collapses the left side of $\cS$ to $(0,0)$ and sends $[0, \delta_1]\times\{1\}$ to $\bigl([0, \delta_1]\times\{1\}\bigr) \cup\bigl( \{0\}\times [0,1]\bigr)$ by the an appropriate piecewise linear homeomorphism fixing $(\delta_1,1)$.  Otherwise $\psi$ is the identity on $\bigl([0, \delta_1]\times\{1\}\bigr) \cup\bigl( \{0\}\times [0,1]\bigr)$.

Since $\psi$ is a quotient map, it induces a map $\bar f: \cS \to X$ such that $\bar f \circ \psi = f$. It is immediate that $\bar f$ is an oscillatory square that satisfies the conclusions of the lemma.

\end{proof}

\begin{lem}\label{thinner}
Let $f:\cS \to X$ be a thick pseudo-oscillatory square in a planar continuum $X$.  Suppose that $k_N, k_S$ are parallel lines in the plane such that $\im\bigl(f|_{\{0,1\}\times[0,1]}\bigr)$ is contained in the thin component of $\mathbb R\backslash \{k_N\cup k_S\}$.  Then $\im(f)$ is contained in the thin component of $\mathbb R\backslash \{k_N\cup k_S\}$.
\end{lem}

\begin{proof}
Suppose that $k_N, k_S$ are parallel lines in the plane such that $\im\bigl(f|_{\{0,1\}\times[0,1]}\bigr)$ is contained in the thin component of $\mathbb R\backslash \{k_N\cup k_S\}$.  Let $2\delta = \d\bigl(\im\bigl(f|_{\{0,1\}\times[0,1]}\bigr), k_N\cup k_S\bigr) >0$.

By Lemma \ref{reparameterization}, there exists a thick oscillatory square $\bar f :\cS \to X$ with the property that $\im\bigl(\bar f|_{\{0,1\}\times[0,1]}\bigr)\subset \mathcal N_\delta \Bigl(\im\bigl( f|_{\{0,1\}\times[0,1]}\bigr)\Bigr)$.  Hence $\im\bigl(\bar f|_{\{0,1\}\times[0,1]}\bigr)$ is contained in the thin component of $\mathbb R\backslash \{k_N\cup k_S\}$ by our choice of $\delta$.

Lemma \ref{thin} implies that the image of $\bar f$ is contained in the thin component of $\mathbb R\backslash \{k_N\cup k_S\}$.  Since $\im(f) = \im(\bar f)$, the image of $f$ is also contained in the thin component of $\mathbb R\backslash \{k_N\cup k_S\}$.
\end{proof}

\begin{lem}\label{weaklyclose}
For any thick pseudo-oscillatory square, $f:\cS \to X$, in a planar continuum $X$ there exists a reparametrization, $\rho : [0,1]\to [0,1]$, such that $$\d\bigl(f(x,1), f(\rho(x), 0)\bigr)< 7\max\Bigl\{\diam\Bigl(\im\bigr(f|_{\{0\}\times[0,1]}\bigr)\Bigr), \diam\Bigl(\im\bigr(f|_{\{1\}\times[0,1]}\bigr)\Bigr)\Bigr\}.$$
\end{lem}

\begin{proof}
    Let $\delta = \max\Bigl\{\diam\Bigl(\im\bigr(f|_{\{0\}\times[0,1]}\bigr)\Bigr), \diam\Bigl(\im\bigr(f|_{\{1\}\times[0,1]}\bigr)\Bigr)\Bigr\}$.  Fix $k_N,k_S$ parallel lines such that $\im\bigr(f|_{\{0,1\}\times[0,1]}\bigr)$ is contained in the thin component of $\mathbb R^2\backslash \{k_N\cup k_S\}$ and $\d(k_N, k_S)< 2\delta$.
	
    Let $l$ be a line which is transverse to each non-degenerate side of $f$ such that $l$ intersects the thin component of $\mathbb R^2\backslash \{k_N\cup k_S\}$ in an open interval of length less than $3\delta$.  (There are only countable many lines that are not transverse to $f$.  Hence, we can find a line that is nearly orthogonal to $k_N$ and is transverse to each nondegenerate side of $f$.)
	
    Let  $\bar f$ be the thick oscillatory square obtained from applying Lemma \ref{reparameterization} to $f$.  Notice $l$ is also transverse to $\bar f$  and we can apply Theorem \ref{malPlanar} to $\bar f$ and $l$ to obtain an upper semicontinuous decomposition $\bar H$  of $\cS$.  Let $H = \{\psi^{-1}(\bar h)\mid \bar h\in\bar H\}$ where $\psi:\cS \to \cS$ is the monotone quotient map from Lemma \ref{reparameterization}.  Then $H$ is an upper semicontinuous decomposition of $\cS$, it is still noncrossing, and each element of $H$ can only intersect the top or the bottom of $\cS$ at a unique point.

    Since $\bar f$ restricted to each element of $\bar H$ maps into a line segment parallel to $l$ and $\im (\bar f)$ is contained in the thin component of $\mathbb R^2\backslash \{k_N\cup k_S\}$, we have that the $\diam\bigl(\bar f|_{\bar h}\bigr) < 3\delta$ for every $\bar h\in \bar H$.  Since $ \bar f\circ\psi= f$, $\diam\bigl( f|_{ h}\bigr) < 3\delta$ for every $ h\in  H$.

    For $s\in\cS$, we will use $h_s$ to denote the element of $H$ containing $s$.

    Let $a = \sup\bigl\{ x\in[0,1] \mid h_{(x,1)}\cap \bigl(\{0\}\times [0,1]\bigr) \neq\emptyset\bigr\}$ and $b = \inf\bigl\{ x\in[0,1] \mid h_{(x,1)}\cap \bigl(\{1\}\times [0,1]\bigr) \neq\emptyset\bigr\}$.  Since $H$ is upper semicontinuous, $h_{(a,1)}$ also intersects the left side of $\cS$ and $h_{(b,1)}$ also intersects the right side of $\cS$. Since $H$ is noncrossing, $a\leq b$.

\begin{figure}[h]
\centering
\def\svgwidth{4in}
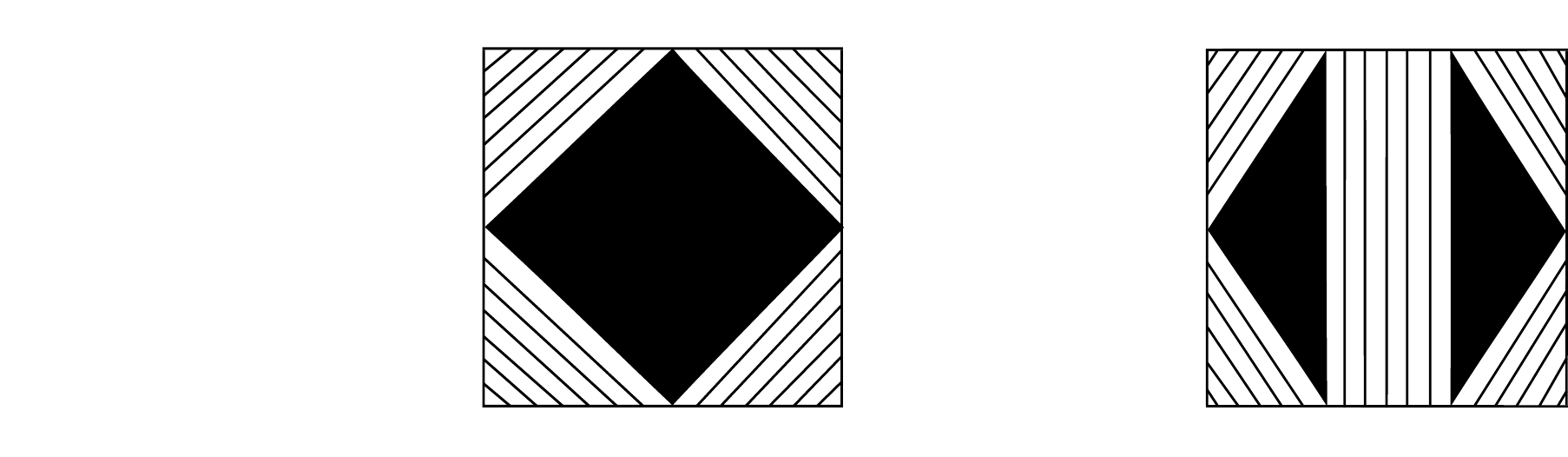
\end{figure}
    \textbf{Case 1: $a = b$.} Then $c = \sup\bigl\{ x\in[0,1] \mid h_{(x,0)}\cap \bigl(\{0\}\times [0,1]\bigr) \neq\emptyset\bigr\} = \inf\bigl\{ x\in[0,1] \mid h_{(x,0)}\cap \bigl(\{1\}\times [0,1]\bigr) \neq\emptyset\bigr\}$.  Let $\rho:[0,1]\to [0,1]$ be the piecewise linear map sending $[0,a] $ to $[0,c]$ and $[a,1]$ to $[c,1]$.

    Suppose $x\in[0,1]$.  If $x\leq a$, both $h_{(x,1)},h_{(\rho(x),1)}$ intersect $\{0\}\times [0,1]$ which implies that $\d\bigl(f(x,1),f(\rho(x),0)\bigr)\leq \d\bigl(f(x,1),f(u)\bigr) + \d\bigl(f(u),f(v)\bigr) + \d\bigl(f(v),f(\rho(x),0)\bigr)< 3\delta + \delta + 3 \delta = 7\delta$ for some point $u,v$ on the left side of $\cS$.  If $x\geq a$ the analogous argument applied to the right side shows the same inequality.

    \textbf{Case 2: $a < b$.}  Then $h_{(a,1)}$ and $h_{(b,1)}$ both intersects the bottom of $\cS$.  Let $(c,0), (d,0)$ be the unique points of intersection of $h_{(a,1)}, h_{(b,1)}$ with the bottom of $\cS$ respectively.  Since $H$ is noncrossing, for ever $x\in[a,b]$, the decomposition element  $h_{(x,1)}$ must intersect the bottom of $\cS$ in a unique point.

    This defines a map $\rho: [a,b]\to [c,d]$ by letting $\rho(x)$ be the unique point of $h_{(x,1)}$ which is contained in the bottom of $\cS$.  Again $H$ is a noncrossing implies that $\rho$ is increasing and since $H$ is a decomposition of $\cS$ the map $\rho$  must be surjective.  Hence $\rho$ is a homeomorphism. We can extend $\rho$ to all of $[0,1]$ be mapping $[0,a]$ linearly to $[0,c]$ and $[b,1]$ to $[d,1]$.  As before $\d\bigl(f(x,1),f(\rho(x),0)\bigr)<7\delta$, since $h_{(x,1)},h_{(\rho(x),0)}$ are either the same element (and hence have diameter at most $3\delta$) or both intersect the same vertical side of $\cS$.

\end{proof}

\begin{lem}\label{close}
    Suppose that $\alpha,\beta: [0,1] \to X$ are $\delta$-homotopic oscillatory geodesics in a planar continuum $X$.  Then there exits a reparametrization $\rho: [0,1] \to [0,1]$ such that $\d\bigl(\alpha(t), \beta\circ\rho(t)\bigr)< 7 \delta$ for all $t\in[0,1]$.
\end{lem}

\begin{proof}
    We can assume that $\alpha,\beta$ are parameterized proportional to total oscillation. Let $f: \cS \to X$ be a pseudo-oscillatory square such that $f(x,1) = \alpha(x)$; $f(x,0)= \beta(x)$ for all $x\in[0,1]$ and $\diam\bigl(\im(h|_{\{i\}\times[0,1]})\bigr)<\delta$ for $i\in \{0,1\}$.  If $f$ has no subdivisions then it is thick and the lemma follows from Lemma \ref{weaklyclose}.

    Otherwise, we may find subdivisions $(s_0,t_0)$ and $(s_1, t_1)$ of $f$ which satisfy the maximality condition of Lemma \ref{subdivisions}.  Since oscillatory geodesics are unique $\alpha|_{[s_0, s_1]}$ is a reparametrization of $\beta|_{[t_0, t_1]}$ and we can define $\rho_1: [s_0, s_1]\to [t_0,t_1]$ such that $f(x,1) = f\bigl(\rho_1(x),0\bigr)$ for all $x\in[s_0,s_1]$.  Since $f$ restricted to the top or the bottom as no constant subpaths, $\rho_1$ can be chosen to be a homeomorphism.

    Since $f(s_0,1)= f(t_0,0)$, we may define a thick pseudo-oscillatory square $f_l$ by $$f_l(x,y) =\begin{cases} f(xs_0, y) &\mbox{ if } y = 1\\
    f(xt_0, y) &\mbox{ if } y = 0\\
    f(x,y)  &\mbox{ if } x = 0\\
    f(s_0,1) &\mbox{ if } x = 1\end{cases}.$$

    By applying Lemma \ref{weaklyclose} to $f_l$, we can find a reparametrization $\bar\rho_l: [0,1]\to [0,1]$ such that $\d\bigl(f_l(x,1), f_l(\bar\rho_l(x),0)\bigr) < 7 \delta$.  Define $\rho_l: [0,s_0]\to [0, t_0]$  by $\rho_l(x) = t_0\bar \rho_l(x/s_0)$ for $x\in[0,s_0]$.

    Notice that $\rho_l(s_0) = \rho_1(s_0)$. For $x\in[0,s_0]$, we have
    \begin{align*}
        \d\bigl(f(x,1), f(\rho_l(x),0)\bigr)  &= \d\bigl(f_l(x/s_0,1), f(t_0\bar\rho_l(x/s_0),0)\bigr)  \\
        & =\d\bigl(f_l(x/s_0,1), f_l(\bar\rho_l(x/s_0),0)\bigr)< 7\delta
    \end{align*}

 We can repeat the argument used to define $\rho_l$ to define a reparametrization $\rho_r: [s_1,1] \to [t_1, 1]$ such that  $\d\bigl(f(x,1), f(\rho_l(x),0)\bigr)< 7\delta$.  Then $\rho: [0,1]\to [0,1]$ given by

        $$\rho(x) =\begin{cases} \rho_l(x) &\mbox{ if } x \leq s_0\\
        \rho_1(x) &\mbox{ if } s_0\leq x \leq s_1\\
        \rho_r(x) &\mbox{ if } x \geq s_1\end{cases}$$
is the desired reparametrization.
\end{proof}

Lemma \ref{close} shows that any two oscillatory geodesics which are $\delta$-endpoint homotopic have parameterizations which are $7\delta$-close.  What we will actually require is that $\delta$-endpoint homotopic paths that are parameterized by total oscillation are $\epsilon$-close.

\begin{prop}\label{continuously}
    Let $X$ be a planar continuum.  For every $\epsilon>0$, there exists a $\delta>0$ such that if two paths are $\delta$-endpoint homotopic, then their oscillatory geodesic representatives parameterized by total oscillation are $\epsilon$-close.
\end{prop}

\begin{proof}
    Notice that for any path $\beta$ and any $0\leq s< t\leq \mathcal T(\beta)$, $$ \mathcal T\bigl(\beta|_{[0,s]}\bigr)+ \mathcal T\bigl(\beta|_{[s,t]}\bigr)\leq \mathcal T\bigl(\beta|_{[0,t]}\bigr).$$

    Fix $\epsilon>0$. We may will choose $0< \delta_1,\delta_2 \leq \epsilon/4$ with the following properties. (Note $\delta_2$ will depend on $\delta_1$.)

    \begin{enumerate}[(i)]

	   \item By Lemma \ref{small diameter} choose $\delta_1$ such that, for any path $\alpha$ in $X$, if $\mathcal T(\alpha)\leq \delta_1$ then $\diam\bigl(\im(\alpha)\bigr) < \epsilon/4$.
	   \item \label{ii} By Lemma \ref{small oscillation difference} choose $\delta_2$ such that, for any two paths $\alpha,\beta:[0,1]\to X$, if $\d\bigl(\alpha(t),\beta(t)\bigr)\leq \delta_2$ for all $t\in[0,1]$ then $|\mathcal T(\alpha)-\mathcal T(\beta)|<\delta_1$.
	
    \end{enumerate}

    Suppose that $\alpha: [0,a]\to X$ and $\beta: [0,b]\to X$ are $\delta$-endpoint homotopic oscillatory geodesics parameterized by total oscillation where $\delta = \min\{\delta_1,\delta_2\}/8$.  Without loss of generality, we can assume that $a\leq b$.

    By Lemma \ref{close}, there exist reparametrizations $\theta_\alpha : [0,1] \to [0,a]$ and $\theta_\beta:[0,1]\to [0,b]$ such that $\d\bigl( \alpha\circ\theta_\alpha(t),\beta\circ\theta_\beta(t)\bigr)< 7\delta< \delta_2$ for all $t\in[0,1]$.  This implies that $|a-b|=|\mathcal T(\alpha) -\mathcal T(\beta)|<\delta_1$.

    Fix $t_0\in\bigl[0,\min\{a,b\}\bigr]$ and choose $s_0\in[0,1]$ such that $\theta_\alpha(s_0) =t_0$.  Then $\bigl|t_0 - \theta_\beta(s_0)\bigr|= \bigl|\mathcal T(\alpha|_{[0,t_0]})- T(\beta|_{[0,\theta_\beta(s_0)]})\bigr| <\delta_1$ by property (\ref{ii}).

    If $ t_0 - \theta_\beta(s_0)\geq 0$, then

    \begin{align*}
        t_0 -\delta_1 + \mathcal T\bigl(\beta|_{\bigl[\theta_\beta(s_0),t_0\bigr]}\bigr) &\leq \theta_\beta(s_0) +  \mathcal T\bigl(\beta|_{\bigl[\theta_\beta(s_0),t_0\bigr]}\bigr) \\
        &= \mathcal T\bigl(\beta|_{\bigl[0,\theta_\beta(s_0)\bigr]}\bigr)+ \mathcal T\bigl(\beta|_{\bigl[\theta_\beta(s_0),t_0\bigr]}\bigr)\leq \mathcal T\bigl(\beta|_{\bigl[0,t_0\bigr]}\bigr)= t_0
    \end{align*}
    which implies that $\mathcal T\bigl(\beta|_{\bigl[\theta_\beta(s_0),t_0\bigr]}\bigr)\leq \delta_1$.  If $ t_0 - \theta_\beta(s_0)<0 $, then
    \begin{align*}
        \theta_\beta(s_0) -\delta_1 + \mathcal T\bigl(\beta|_{\bigl[t_0,\theta_\beta(s_0)\bigr]}\bigr) &\leq t_0+  \mathcal T\bigl(\beta|_{\bigl[t_0,\theta_\beta(s_0)\bigr]}\bigr)\\
        &= \mathcal T\bigl(\beta|_{\bigl[0,t_0\bigr]}\bigr)+ \mathcal T\bigl(\beta|_{\bigl[t_0,\theta_\beta(s_0)\bigr]}\bigr)\leq \mathcal T\bigl(\beta|_{\bigl[0,\theta_\beta(s_0)\bigr]}\bigr)=\theta_\beta(s_0)
    \end{align*}
    which implies that $\mathcal T\bigl(\beta|_{\bigl[t_0, \theta_\beta(s_0)\bigr]}\bigr)\leq \delta_1$.  In both cases, our choice of $\delta_1$ implies the we have the following inequality $\d\bigl(\beta(t_0), \beta\circ\theta_\beta(s_0)\bigr) <\epsilon/4$.

    Therefore $\d\bigl(\alpha(t_0), \beta(t_0)\bigr) \leq \d\bigl(\alpha(t_0), \beta\circ\theta_\beta(s_0)\bigr) + \d\bigl( \beta\circ\theta_\beta(s_0), \beta(t_0)\bigr)<\epsilon/2$ for all $t_0\in\bigl[0,\min\{a,b\}\bigr]$.

    We can now consider the case $t_0\in(a,b]$.  Notice $a + \mathcal T\bigl(\beta|_{\bigl[a,b\bigr]}\bigr)= \mathcal T\bigl(\beta|_{\bigl[0,a\bigr]}\bigr)+ \mathcal T\bigl(\beta|_{\bigl[a,b\bigr]}\bigr)\leq \mathcal T\bigl(\beta\bigr)= b$, which proves $\diam\bigl(\beta|_{[a,b]}\bigr)<\epsilon /4$.  Thus $\d\bigl(\alpha(a), \beta(t_0)\bigr) \leq \d\bigl(\alpha(a), \beta(a)\bigr) + \d\bigl( \beta(a), \beta(t_0)\bigr)<\epsilon$.

\end{proof}

\section{The fundamental group determines homotopy type}

\begin{defn}
    Let $\alpha: \bigl([0,1], 0,1\bigr) \to \bigl(X, x_0, x_1\bigr)$ be a path.  We will use $\overline \alpha$ to denote the path $\overline \alpha (t) = \alpha(1-t)$.  The path $\alpha$ induces a change of base point isomorphism $\widehat \alpha : \pi_1(X, x_0) \to \pi_1(X,x_1)$ defined by $\widehat{\alpha}([s]) = [\overline\alpha* s* \alpha]$.  Given continuous maps $f,g: Y\to X$; we will say that $f$ and $g$ are \emph{conjugate by a path} $\alpha: (I,0,1) \to \bigl(X, g(y_0), f(y_0)\bigr)$ if $f_*$ and $\widehat \alpha \circ g_*$ give the same homomorphism from $\pi_1(Y,y_0)$ to $\pi_1\bigl(X, f(y_0)\bigr)$.
\end{defn}

\begin{lem}\label{classical}[A classical lemma]
    Suppose that $f,g: Y\to X$ are homotopic maps, i.e. there exists $ F:Y\times I \to X$ such that $F(y,0) = g(y)$ and $F(y,1) = f(y)$.  For every $y\in Y$ there exists a path $\alpha_y$ from $g(y)$ to $f(y)$ such that $f$ and $g$ are conjugate by $\alpha_y$.
\end{lem}

    Notice that  $\alpha_y (t) =F(y,t)$ is the desired path.  The following lemma allows us to find a suitable path to prove a converse to Lemma \ref{classical} for planar and one-dimensional Peano continua.

\begin{lem}\label{map induced paths}
    Suppose that $f,g:Y\to X$ are continuous maps that are conjugate by a path $\alpha: (I,0,1) \to \bigl(X, g(y_0), f(y_0)\bigr)$. Let $\beta: \bigl([0,1],0,1\bigr) \to \bigl(Y, y_0,y\bigr)$ be a path. Then $f\circ\overline\beta*\overline\alpha*g\circ\beta$ is a path from $f(y)$ to $g(y)$ whose homotopy class is independent of $\beta$.
\end{lem}

We will say that any path from $f(y)$ to $g(y)$ in the homotopy class of $f\circ \overline\beta * \overline\alpha * g\circ\beta$ is \emph{induced by $f \stackrel{\alpha}{\rightarrow} g$}.

\begin{proof}
    Let $\beta_i: \bigl([0,1],0,1\bigr) \to \bigl(Y, y_0,y\bigr)$ be a path from $y_0$ to $y$ for $i = 1,2$.  Then $\beta_1*\overline\beta_2$ is a loop at $y_0$  which implies that $f\circ\beta_1*f\circ\overline\beta_2$ is homotopic relative to endpoints to $ \overline\alpha*g\circ\beta_1*g\circ\overline\beta_2*\alpha$.  This implies that $f\circ\overline\beta_1*\overline\alpha*g\circ\beta_1$ is homotopic to $f\circ\overline\beta_2*\overline\alpha*g\circ\beta_2.$
\end{proof}

\begin{lem}\label{close paths}
    Let $f,g:Y \to X$ be continuous maps  from a Peano continuum into a planar or one-dimensional Peano continuum.  Suppose that $f$ and $g$ are conjugate by a path $\alpha: (I,0,1) \to \bigl(X, g(y_0), f(y_0)\bigr)$.  Then the oscillatory geodesics induced by $f\stackrel{\alpha}{\rightarrow} g$ vary continuously.

    To be precise, for every $\epsilon>0$ there exists a $\delta>0$ such that, for every $y_1, y_2\in Y$, the oscillatory geodesics induced by $f\stackrel{\alpha}{\rightarrow} g$ from $f(y_1)$ to $g(y_1)$ and from $f(y_2)$ to $g(y_2)$ are $\epsilon$-close whenever $\d(y_1,y_2)< \delta$.
\end{lem}

\begin{proof}
    Fix $\epsilon>0$. By Proposition \ref{continuously} (when $X$ is planar) or \cite[Theorem 3.9]{ccz} (when $X$ is one-dimensional), we may choose a $\delta>0$ such that if two paths are $\delta$-endpoint homotopic, then their oscillatory geodesic representatives parameterized by total oscillation are $\epsilon$-close.  Since $f$ and $g$ are continuous and $Y$ is compact, there exists a $\delta_1>0$ with the property that $\diam \bigl(f\bigl(\Ball_{\delta_1}(y)\bigr)\bigr),\diam \bigl(g\bigl(\Ball_{\delta_1}(y)\bigr)\bigr)<\delta$ for every $y\in Y$.  Since $Y$ is locally path connected there exists a $\delta_0$ such that $\Ball_{\delta_0}(y)$ is path connected in $\Ball_{\delta_1}(y)$ for every $y\in Y$.

    Suppose that $y_2\in \Ball_{\delta_0}(y_1)$.  Let $\beta$ be a path from $y_0$ to $y_1$ and $\gamma$ be a path from $y_1$ to $y_2$ contained in $\Ball_{\delta_1}(y_1)$.  Then $f\circ\overline\beta*\overline\alpha*g\circ \beta$  and $f\circ\overline\gamma*f\circ\overline\beta* \overline\alpha*g\circ \beta*g\circ\gamma$ are paths from $f(y_1)$ to $g(y_1)$ and $f(y_2)$ to $g(y_2)$ respectively which are in the homotopy classes induced by $f\stackrel{\alpha}{\rightarrow} g$.  By construction these two paths are $\delta$-endpoint homotopic, since $\max\bigl\{\diam\bigl(\im(g\circ\gamma)\bigr),  \diam\bigl(\im(f\circ\gamma)\bigr)\bigr\}< \delta$.  Thus their oscillatory geodesic representatives are $\epsilon$-close by our choice of $\delta$.
\end{proof}

\begin{thm}\label{homotopic maps}
    Let $f,g:Y\to X$ be continuous maps from a Peano continuum into a planar or one-dimensional Peano continuum.  If $f$ and $g$ are conjugate by a path, then $f$ and $g$ are homotopic maps.

\end{thm}
\begin{proof}
    Fix a path $\alpha: (I,0,1) \to \bigl(X, g(y_0), f(y_0)\bigr)$ such that $f_*$ and $\widehat \alpha \circ g_*$ give the same homomorphism from $\pi_1(Y,y_0)$ to $\pi_1\bigl(X, f(y_0)\bigr)$.  Define $F: Y\times I \to X$ by $F(y,t) = \alpha_y\Bigl(\theta\bigl(\mathcal T(\alpha_y),t\bigr)\Bigr)$ where $\alpha_y$ is the oscillatory geodesic from $f(y)$ to $g(y)$ in the path class induced by $f\stackrel{\alpha}{\rightarrow} g$ and $\theta(s,t) =\min\{s,t\}$.

    Fix $(y_1,t_1)\in Y\times I$ and $\epsilon> 0$.  By Lemma \ref{close paths}, we may choose a $\delta_1 >0$ such that $\alpha_{y_1}$ and $\alpha_{y}$ are $\epsilon/2$-close for every $y\in \Ball_{\delta_1}(y_1)$.  Lemma \ref{small diameter} allows us to choose a $\delta_2>0$ such that $\mathcal T(\beta) < \delta_2$ implies that $\diam\bigl(\im(\beta)\bigr)\leq \epsilon/2$.  It is now a straight forward exercise to see that $\d\bigl(F(y_1,t_1), F(y,t)\bigr)< \epsilon$ for every $(y,t)\in \Ball_{\delta_1}(y_1)\times \bigl[\max\{0, t_1-\delta_2\}, \min \{1, t_1+ \delta_2\}\bigr]$.  Thus $F$ is continuous.
\end{proof}

\begin{cor}
    Every map for a connected Peano continuum with a torsion fundamental group into a planar or one-dimensional space is nullhomotopic.
\end{cor}

\begin{proof}
Let $Y$ be a connected Peano continuum with torsion fundamental group and $f: Y\to X$ be any continuous map into a planar or one-dimensional space. We will assume that $X$ is a Peano continuum by passing to the subset $f(Y)$.  Then $f_*$ is trivial, since the fundamental groups of one-dimensional or planar Peano continua are locally free.  Theorem \ref{homotopic maps} implies that $f$ is homotopic to the identity.
\end{proof}

\begin{cor}
    Every map for a simply connected Peano continuum into a planar or one-dimensional space is nullhomotopic.
\end{cor}

We will need the following lemma proved by Conner and Kent.

\begin{prop}\label{inverse}[\cite[Lemma 3.14]{ConnerKentpreprint2}]
    If $X,Y$ are planar or one-dimensional Peano continuum with isomorphic fundamental groups, then there exists continuous maps $f: X \to Y$ and $g: Y\to X$ such that $g\circ f(x_0) = x_0$  and $(g\circ f)_*$, $(f\circ g)_*$ are the identity homomorphisms of the respective groups.
\end{prop}

Note that Lemma 3.14 of \cite{ConnerKentpreprint2} only specifically states that $(g\circ f)_*$ is the identity homomorphism.  However, the proof shows that $g_*$ is the inverse isomorphism to $f_*$ and hence $(f\circ g)_*$ is also the identity isomorphism.

\setcounter{section}{1}\setcounter{thm}{0}
\begin{thm}\label{homotopytype}
    If $X$ and $Y$ are one-dimensional or planar Peano continuum such that $\pi_1(X,x_0)$ is isomorphic to $\pi_1(Y,y_0)$ for some choice of $x_0$ and $y_0$, then $X$ and $Y$ are homotopy equivalent.
\end{thm}

\begin{proof}
    By Proposition \ref{inverse}, we can find maps $f: X \to Y$ and $g: Y\to X$ such that $g\circ f(x_0) = x_0$  and $(g\circ f)_*$, $(f\circ g)_*$ are the identity homomorphisms on the corresponding fundamental groups.  Hence $f\circ g$ and $g\circ f$ are conjugate  by the trivial path to the identity map and Theorem \ref{homotopic maps} implies that $f\circ g $ and $g\circ f$ are homotopic to the identity on $Y$ and $X$ respectively.  Thus, $f$ and $g$ are homotopy equivalences.
\end{proof}

\bibliographystyle{plain}
\bibliography{../../../bib}

\end{document}

%% file: decomp.pdf_tex
\begingroup%
  \makeatletter%
  \providecommand\color[2][]{%
    \errmessage{(Inkscape) Color is used for the text in Inkscape, but the package 'color.sty' is not loaded}%
    \renewcommand\color[2][]{}%
  }%
  \providecommand\transparent[1]{%
    \errmessage{(Inkscape) Transparency is used (non-zero) for the text in Inkscape, but the package 'transparent.sty' is not loaded}%
    \renewcommand\transparent[1]{}%
  }%
  \providecommand\rotatebox[2]{#2}%
  \ifx\svgwidth\undefined%
    \setlength{\unitlength}{518.45974659bp}%
    \ifx\svgscale\undefined%
      \relax%
    \else%
      \setlength{\unitlength}{\unitlength * \real{\svgscale}}%
    \fi%
  \else%
    \setlength{\unitlength}{\svgwidth}%
  \fi%
  \global\let\svgwidth\undefined%
  \global\let\svgscale\undefined%
  \makeatother%
  \begin{picture}(1,0.21477003)%
    \put(0,0){\includegraphics[width=\unitlength,page=1]{decomp.pdf}}%
    \put(0.5242443,0.0934791){\color[rgb]{0,0,0}\makebox(0,0)[lb]{\smash{$k$}}}%
    \put(0.69154867,0.18138447){\color[rgb]{0,0,0}\makebox(0,0)[lb]{\smash{$f(a)$}}}%
    \put(0.90432903,0.18113682){\color[rgb]{0,0,0}\makebox(0,0)[lb]{\smash{$f(b)$}}}%
    \put(0,0){\includegraphics[width=\unitlength,page=2]{decomp.pdf}}%
    \put(0.42140822,0.11229777){\color[rgb]{0,0,0}\makebox(0,0)[lb]{\smash{$f$}}}%
    \put(0,0){\includegraphics[width=\unitlength,page=3]{decomp.pdf}}%
    \put(0.52401762,0.15127776){\color[rgb]{0,0,0}\makebox(0,0)[lb]{\smash{$k'$}}}%
    \put(0,0){\includegraphics[width=\unitlength,page=4]{decomp.pdf}}%
    \put(0.08942433,0.00753793){\color[rgb]{0,0,0}\makebox(0,0)[lb]{\smash{$a_0$}}}%
    \put(0.20146686,0.00407986){\color[rgb]{0,0,0}\makebox(0,0)[lb]{\smash{$b_0$}}}%
    \put(0.19976165,0.19695886){\color[rgb]{0,0,0}\makebox(0,0)[lb]{\smash{$b_1$}}}%
    \put(0.09036631,0.19695886){\color[rgb]{0,0,0}\makebox(0,0)[lb]{\smash{$a_1$}}}%
    \put(0,0){\includegraphics[width=\unitlength,page=5]{decomp.pdf}}%
    \put(0.34213402,0.1847057){\color[rgb]{0,0,0}\makebox(0,0)[lb]{\smash{$f^{-1}(k')$}}}%
    \put(0,0){\includegraphics[width=\unitlength,page=6]{decomp.pdf}}%
  \end{picture}%
\endgroup%

%% file: mal.pdf_tex
\begingroup%
  \makeatletter%
  \providecommand\color[2][]{%
    \errmessage{(Inkscape) Color is used for the text in Inkscape, but the package 'color.sty' is not loaded}%
    \renewcommand\color[2][]{}%
  }%
  \providecommand\transparent[1]{%
    \errmessage{(Inkscape) Transparency is used (non-zero) for the text in Inkscape, but the package 'transparent.sty' is not loaded}%
    \renewcommand\transparent[1]{}%
  }%
  \providecommand\rotatebox[2]{#2}%
  \ifx\svgwidth\undefined%
    \setlength{\unitlength}{520.1095557bp}%
    \ifx\svgscale\undefined%
      \relax%
    \else%
      \setlength{\unitlength}{\unitlength * \real{\svgscale}}%
    \fi%
  \else%
    \setlength{\unitlength}{\svgwidth}%
  \fi%
  \global\let\svgwidth\undefined%
  \global\let\svgscale\undefined%
  \makeatother%
  \begin{picture}(1,0.29099564)%
    \put(0,0){\includegraphics[width=\unitlength,page=1]{mal.pdf}}%
    \put(0.21404369,-0.03041349){\color[rgb]{0,0,0}\makebox(0,0)[lb]{\smash{Case 1}}}%
    \put(0.82929866,-0.03041349){\color[rgb]{0,0,0}\makebox(0,0)[lb]{\smash{Case 2}}}%
    \put(0,0){\includegraphics[width=\unitlength,page=2]{mal.pdf}}%
    \put(0.07208442,0.27324098){\color[rgb]{0,0,0}\makebox(0,0)[lb]{\smash{$a=b$}}}%
    \put(0.38740263,0.27324098){\color[rgb]{0,0,0}\makebox(0,0)[lb]{\smash{$a=b$}}}%
    \put(0.83192431,0.27324098){\color[rgb]{0,0,0}\makebox(0,0)[lb]{\smash{$a$}}}%
    \put(0.91806001,0.27324098){\color[rgb]{0,0,0}\makebox(0,0)[lb]{\smash{$b$}}}%
    \put(0.09823276,0.00406692){\color[rgb]{0,0,0}\makebox(0,0)[lb]{\smash{$c$}}}%
    \put(0.41047469,0.00406692){\color[rgb]{0,0,0}\makebox(0,0)[lb]{\smash{$c$}}}%
    \put(0.83192431,0.00406692){\color[rgb]{0,0,0}\makebox(0,0)[lb]{\smash{$c$}}}%
    \put(0.91806001,0.00406692){\color[rgb]{0,0,0}\makebox(0,0)[lb]{\smash{$d$}}}%
  \end{picture}%
\endgroup%